\documentclass{amsart}
\usepackage{color,enumitem,graphicx}
\usepackage{graphicx}
\usepackage{subcaption}
\usepackage{amsthm}
\usepackage{amsmath,amssymb}
\usepackage{mathrsfs}
\usepackage[breaklinks]{hyperref}
\newtheorem{thm}{Theorem}[section]

\newtheorem{lem}[thm]{Lemma}
\newtheorem{prop}[thm]{Proposition}
\theoremstyle{definition}
\newtheorem{defn}[thm]{Definition}
\theoremstyle{remark}
\newtheorem{rem}[thm]{Remark}

\numberwithin{equation}{section}
\newcommand{\abs}[1]{\left\vert#1\right\vert}

\newcommand{\norm}[1]{\left\Vert#1\right\Vert}

\begin{document}

\title[ Asymptotic behaviour of the least energy solutions]
 {Asymptotic behaviour of the least energy solutions of fractional semilinear Neumann problem}
\author[S.\ Gandal, J.\,Tyagi]
{ Somnath Gandal, Jagmohan Tyagi }
\address{Somnath Gandal \hfill\break
 Indian Institute of Technology Gandhinagar \newline
 Palaj, Gnadhinagar Gujarat India-382355.}
\email{gandal$\_$somnath@iitgn.ac.in}
\address{Jagmohan\,Tyagi \hfill\break
 Indian Institute of Technology Gandhinagar \newline
 Palaj, Gandhinagar Gujarat, India-382355.}
 \email{jtyagi@iitgn.ac.in, jtyagi1@gmail.com}
%\date{-----}
%\thanks{Submitted ----.  Published-----.}
\thanks{Submitted \today   \,\,\,\,\,\,\,\,Published-----.}
\subjclass[2010]{35J60, 35B09, 35B40, 35J61, 35R11, 35D30.}
\keywords{Semilinear Neumann problem; fractional Laplacian; positive solutions; asymptotic behaviour}

\begin{abstract}
We establish the asymptotic behaviour of the least energy solutions of the following nonlocal Neumann problem:
 \begin{align*}	\left\{\begin{array}{l l} { d(-\Delta)^{s}u+ u= \abs{u}^{p-1}u }  \text{ in $\Omega,$  } \\ 
		 { \mathcal{N}_{s}u=0 }  \text{ in $\mathbb{R}^{n}\setminus \overline{\Omega},$} \\
		 {u>0}  \text{ in $\Omega,$} \end{array} \right.\end{align*}
where  $\Omega \subset \mathbb{R}^{n}$ is a bounded domain of class $C^{1,1}$, $1<p<\frac{n+s}{n-s},\,n> \max \left\{1, 2s \right\}, 0<s<1,\,d>0$  and $\mathcal{N}_{s}u$ is the nonlocal Neumann derivative. We show that for small $d,$ the least energy solutions $u_d$ of the above problem achieves $L^{\infty}$ bound independent of $d.$  Using this together with suitable $L^{r}$-estimates on $u_d,$ we show that least energy solution $u_d$ achieve maximum on the boundary of  $\Omega$ for $d$ sufficiently small.

\end{abstract}

\maketitle

\tableofcontents
%\section{Notations}
%\noindent$\Omega \subset \mathbb{R}^{n}.$ \\
%$\mathcal{C} \Omega:= \mathbb{R}^n\setminus \Omega.$\\
%$B_{r}(x):=$ Open ball of radius $r$ with center at $x$ in $\mathbb{R}^{n}.$ \\
%$\rho(A,B):=$ Distance between subsets $A$ and $B$ of $\mathbb{R}^n.$
\section{Introduction}
We discuss the asymptotic behaviour of non-constant least energy solutions of the following problem:
 \begin{align}\label{P1}\left\{\begin{array}{l l} { d(-\Delta)^{s}u+ u= \abs{u}^{p-1}u }  \text{ in $\Omega,$  } \\ 
 		 { \mathcal{N}_{s}u=0 }  \text{ in $\mathcal{C} \overline{\Omega},$} \\
 		 {u>0}  \text{ in $\Omega,$} \end{array} \right.\end{align}
 where $\Omega \subset \mathbb{R}^{n}$ be a bounded domain of class $C^{1,1},\,\,1<p<\frac{n+s}{n-s}, n> \max \left\{1, 2s \right\}, 0<s<1, d>0,\,
\mathcal{C} \Omega:= \mathbb{R}^n\setminus \Omega$ and $\mathcal{N}_{s}u$ is the nonlocal Neumann derivative, which is defined next.
The nonlocal operator $(-\Delta)^{s}$ is called the fractional Laplacian which is defined as follows:
\begin{align}\label{pointwisedefn}
	(-\Delta)^{s}u(x)= c_{n,s}P.V.\int_{\mathbb{R}^{n}}\frac{u(x)-u(y)}{\abs{x-y}^{n+2s}}dy.
\end{align}
 Here, by P.V., we mean the Cauchy principal value and $c_{n,s}$ is a normalizing constant, given by
 $$c_{n,s}=\left(\int_{\mathbb{R}^{n}} \frac{1-cosx_{1}}{\abs{x}^{n+2s}}dx\right)^{-1},$$
 see for instance \cite{Di} for the details.  Recently, Dipierro et al. \cite{Dip} have introduced a new nonlocal Neumann condition $\mathcal{N}_{s}, $ which is defined as follows:
 \begin{align}\label{neumann}
 	 \mathcal{N}_{s}u(x):=c_{n,s}\int_{\Omega}\frac{u(x)-u(y)}{\abs{x-y}^{n+2s}}dy, \,\, x \in \mathcal{C} \overline{\Omega}. \end{align} 
The advantage of this nonlocal Neumann condition is that it has simple probabilistic interpretation and (\ref{P1}) has a variational structure. 
Further, $ \mathcal{N}_{s}u$ approaches to the classical Neumann derivative $\partial_{\nu}u$ as $s$ goes to $1.$ 

In the last few decades, mathematical analysis of biological phenomena has gained much attention. For example, the chemotaxis models, which are also known as Keller-Segel models \cite{Kell}, have been widely studied in different directions in many papers, see \cite{Chi, Cie, Dj,  Her, Hill, Hill1, Hors, Jag, Nag} and  the reference therein. Chemotaxis is the natural behaviour of an organism in response of  surrounding chemical gradients that are frequently separated by the cells themselves. We refer to \cite{Arm, Hors, Hors1} for a survey on this subject. The Keller-Segel system with suitable initial data has blow-up solutions in dimension $n \geq 2$ and all solutions are regular in dimension $n=1,$ see for instance \cite{Hill, Hors2, Jag, Nag} and the references therein. The analysis on the steady-state for a chemotactic aggregation model with linear or logarithmic sensitivity function was thoroughly done in many papers, see for instance \cite{Kab,Lin1,Nanj,Ni1,Sch}. Let us point out that the following semilinear Neumann problem is an example of Keller-Segel model with a logarithmic chemotactic sensitivity:
 \begin{align}\label{P3}\left\{\begin{array}{l l} {-d \Delta u+ u= \abs{u}^{p-1}u }  \text{ in $\Omega,$  } \\ 
 		\hspace{0.9cm} {\frac{\partial u}{\partial \nu}=0 }  \text{ on $\partial \Omega$} \\
 		\hspace{1.1cm} {u>0}  \text{ in $\Omega,$} \end{array} \right.\end{align}
 where $d>0,$ $\Omega \subset \mathbb{R}^n$ is a bounded domain with smooth boundary and $1<p \leq \frac{n+2}{n-2}$ if $n \geq 3$ and $1<p<\infty$ if $p=2,$ see \cite{Lin1, Sch} for the details. Problem (\ref{P3}) admits a non-constant solution for $d$ sufficiently small, see\,\cite{Adi1,Lin1,Lin2}. Lin et al.\,\cite{Lin1} and C. S. Lin, W. -M. Ni \cite{Lin2} established the solutions of \eqref{P3} in the subcritical case
 $1<p < \frac{n+2}{n-2} .$ In the critical case, when $ p= \frac{n+2}{n-2},$ Adimurthi and G. Mancini \cite{Adi1} obtained a solution of (\ref{P3}).  There have been developments on the asymptotic behaviour of solutions to such equations. In the subcritical case,  $1<p < \frac{n+2}{n-2},$ W. -M. Ni and I. Takagi \cite{Ni1, Ni2} have studied the shape of least energy solutions of (\ref{P3}). They have shown that the least energy solutions tends to zero as the diffusion constant $d$ goes to zero except at finite number of points. Moreover, the maximum of a solution $u_{d}$ of (\ref{P3}) is attained at a unique point on the boundary of  $\Omega.$ The critical case, i.e., $p= \frac{n+2}{n-2},$ was examined by Adimurthi et al. \cite{Adi} using the blow-up analysis. We refer to \cite{he} for the existence, non-existence and the asymptotic behaviour to critical fractional Choquard equation with a local perturbation.  \\

 %They were motivated from the fact that problems like (\ref{P3}) arise in the pattern formation in various models of mathematical biology, for example, Keller-Segal model \cite{Kell}. 
 
We mention that Problem (\ref{P1}) which we explore in this paper is a nonlocal analogue of the classical problem (\ref{P3}).\\

Recall that the movements of cells of some organisms cannot be described by random jumps. In such situations, L\'{e}vy flights plays an important role. The generalized Keller-Segel model with nonlocal diffusion term $d(-\Delta)^{s}$, where $d$ is a positive constant is used to investigate the  chemotaxis.  For the fractional Keller-Segel model, we refer to \cite{Esc,Huan}. In \cite{Huan}, H. Huang and J. Liu studied the existence, stability, uniqueness and regularity for the following model in dimension $n \geq 2:$ 
\begin{align}
	\begin{cases}
		u_{t}= d(-\Delta)^{s} u- \nabla \cdot \left(u \nabla \phi \right), & x \in \mathbb{R}^n, \,\,\,t \geq 0,  \\ -\Delta \phi=u, & \\ u(x,0)=u_{0}(x),
	\end{cases}
\end{align}
where $d$ is a positive constant, $u(t, x)$ is the density of some biological cells and $\phi(t, x)$ is the chemical substance concentration. We mention the work \cite{cho}, where authors have investigated the asymptotic behaviour of solutions for nonlinear elliptic problems for fractional Laplacian with Dirichlet boundary condition. We refer to \cite{mol}  and the reference therein for in-depth treatment of variational methods to nonlocal fractional problems.\\

Motivated by the above works and very recent works on nonlocal Neumann problem for fractional Laplacian and its connections with fractional Keller-Segel models, we have the following natural question to ask:\\

Question. Can we establish the asymptotic behaviour of least energy solutions of \eqref{P1}?\\

The aim of this paper is to answer the above question. A weak solution of (\ref{P1}) can be obtained as a critical point of the energy functional $J_{d},$ which is defined as follows:
\begin{align}
	J_{d}(u):=\frac{1}{2} \Big[ \frac{dc_{n,s}}{2} \int_{T(\Omega)} \frac{\abs{u(x)-u(y)}^{2}}{\abs{x-y}^{n+2s}}dxdy + 
	\int_{\Omega}u^{2}dx \Big] -\frac{1}{p+1}\int_{\Omega}{\lvert u \rvert}^{p+1}dx, \,\,\,\, u \in H_{\Omega}^{s}.
\end{align}
In the above equation $T(\Omega)= \mathbb{R}^{2n}\setminus (\mathcal{C} \Omega)^{2}$ and the space $H_{\Omega}^{s}$ is defined in (\ref{space}). The functional $J_{d}$ is well-defined and of class $C^2$ follows from the Theorem \ref{fembed}.
An application of \textit{Mountain-Pass Lemma} applying to the functional $J_{d}$ yields that
% Let $u_{d} \geq 0$ be a nonconstant weak solution to (\ref{P1}) obtained by using the mountain-pass theorem and 
\begin{align}\label{critical1}
	c_{d}:=\inf_{\gamma \in \Gamma}\max_{[0,1]}J_{d}(\gamma(t))
\end{align}
is a critical value of $J_{d}.$ In the above equation, by $\Gamma,$ we mean the following set:
\begin{align*}
	\Gamma = \biggl\{\gamma \in C([0,1]; H^{s}_{\Omega}) \mid \gamma(0)=1, \, \gamma(1)=u \biggr\},
\end{align*}
where $u \in H^{s}_{\Omega},$ $u>0$ and satisfying $J_{d}(u)=0.$ It turns out that $c_d$ is the least positive critical value, see, Lemma \ref{critical2} next.
For the details one may refer, Theorem 6.1 \cite{Bar} and Theorem 1.1 \cite{Chen1}, where authors have obtained a nonnegative weak solution $u_{d}$ of  (\ref{P1}) with critical value $c_{d}$, provided $d$ is sufficiently
	small. Moreover, $u_{d}$ satisfies $$0< J_{d}(u_{d}) \leq C d^{\frac{n}{2s}},$$
where the constant $C$ is independent of $d.$ Consequently, $u_{d}$ is non-constant. From the proof of Theorem 1.1 \cite{Chen1}, it is immediate to see that  the critical points of $J_{d}$ are not sign changing in $\Omega$. In fact, when $u_{d} \leq 0,$ we can choose $-u_{d}$ in order to have a nonnegative solution of (\ref{P1}). By the strong maximum principle (see, Theorem 2.6 \cite{Cin}), one can see that $u_{d} >0$ a.e. in $\Omega.$ Further, since $u_{d}$ satisfies the Neumann condition $\mathcal{N}_{s}u_{d}(x)=0$ in $\mathcal{C} \Omega$ which implies that $u_{d} >0$ a.e. in  $\mathbb{R}^{n}.$

\begin{defn}
		We call a critical point $u_{d}$ of $J_{d}$ with $J_{d}(u_{d})=c_{d},$ the \textit{least energy} solution or \textit{Mountain-Pass} solution of (\ref{P1}).
\end{defn}

We show the asymptotic behaviour of least energy solutions of (\ref{P1})  following the similar approach as was used for (\ref{P3}) by W. -M. Ni and I. Takagi \cite{Ni1}. 
They used a positive solution $w$ of non-linear Schr\"{o}dinger equation 
$$- \Delta u + u= \abs{u}^{p-1}u \text{ in } \mathbb{R}^{n}, \,\,\,\,\,1<p< \frac{n+2}{n-2}$$ 
to study the asymptotic behaviour of the least energy solutions of (\ref{P3}). The fractional non-linear Schr\"{o}dinger equation 
\begin{align} \label{P2}
	(-\Delta)^{s}u + u = \abs{u}^{p-1}u \text{ in } \mathbb{R}^{n},
\end{align}
where  $1<p<\frac{n+2s}{n-2s}, n> \max \left\{1, 2s\right\}, 0<s<1$ is thoroughly studied, see for instance \cite{Chen,Dip1,FelQ, FelW} and the references therein. 

The main idea of this work is as follows. Let $c_{d}$ be a critical value of $J_{d},$ which is defined in (\ref{critical1}). We use a positive solution $w$ of (\ref{P2})  to observe the asymptotic behaviour of $c_{d}$ as $d\downarrow 0.$ More specifically, $w$ is used to build a suitable function $\phi_d$ to compare $c_d$ with $\max_{t\geq 0} J_{d}(t \phi_{d}).$
In particular, we obtain an inequality $$c_{d}< \frac{d^{\frac{n}{2s}}}{2}F(w)$$ for $d$ sufficiently small, where $F$ is the functional associated with \eqref{P2}, defined in \eqref{energy2}.  
This is closely related to the location of maximum point of a solution $u_{d}$ of (\ref{P1}) on the boundary of  $\Omega .$ \\\\
Now, we summarise the above discussions in terms of the following three main theorems:
A priori it is known that for $1 \leq p < \frac{n+s}{n-s},$ any weak solution $u$ of (\ref{P1}) satisfies $$\norm{u}_{L^{\infty}(\Omega)} \leq K,$$ where $K>0$ is some constant depending on $\Omega, p$ and $d,$ see Theorem 3.1\cite{Mug2}. In next result we obtain a bound for least energy solution $u_{d}$ of (\ref{P1}) which is independent of $d.$

\begin{thm}\label{bound1}
	Let $u_{d}$ be the least energy solution of (\ref{P1}). Then 
	\begin{align}\label{bound1.1}
		d	\frac{c_{n,s}}{2} \int_{T(\Omega)} \frac{\abs{u_{d}(x)-u_{d}(y)}^{2}}{\abs{x-y}^{n+2s}}dxdy + 
		\int_{\Omega}u_{d}^{2}dx=& \int_{\Omega} u_{d}^{p+1}dx \leq C_{0}d^{\frac{n}{2s}},
	\end{align}
	where $C_{0}>0$ is some constant depending on $p.$
	Moreover, there is a constant $C_{1}>0$ depending only on $p$ and $\Omega$ such that \begin{align}\label{bound1.2}
		\sup_{\Omega}u_{d}(x)\leq C_{1}.
	\end{align}
\end{thm}
In the next theorem, we show that the $L^r$-norm of the least energy solution $u_d$ is bounded by $d^{\frac{n}{2s}}$ times some constant independent of $d.$
\begin{thm}\label{lr-estimate}
	Let $u_{d}$ be the least energy solution of (\ref{P1}). Then  \begin{align}\label{lr1.1}
		b(r)d^{\frac{n}{2s}}\leq \int_{\Omega} u_{d}^{r}dx \leq B(r)d^{\frac{n}{2s}}, \,\,\, \text{if } 1\leq r \leq \infty. \\
		\label{lr1.4}   b(r)d^{\frac{n}{2s}}\leq \int_{\Omega} u_{d}^{r}dx \leq B(r)d^{\frac{nr}{2s}}, \,\,\, \text{if } 0<r<1,
	\end{align}
	where $b(r)$ and $B(r)$ are positive constants such that $b(r)<B(r)$ and are independent of $d.$
\end{thm}
%Let $\Omega \subset \mathbb{R}^{n}_{+}$ be a bounded domain of class $C^{1}$ which contains $B_{r}^{+}$ for some $r>0.$
 We show the asymptotic behaviour in next theorem.
 \begin{thm}\label{dist}
		Let $\Omega \subset \mathbb{R}^n$ be a bounded domain of class $C^{1,1}.$ 
		Let $u_{d}$ be the least energy solution of (\ref{P1}). If $u_{d}$ achieves maximum at a point $z_{d} \in \overline{\Omega},$ then for all $d$ sufficiently small, we have the following:
		\begin{enumerate}
			\item[(A)] There exists a positive constant $K_{*}$ such that $\rho (z_{d}, \partial \Omega) \leq K_{*}d^{\frac{1}{2s}} .$
			%\begin{align}\label{dist1.1}
			%\rho (z_{d}, \partial \Omega) \leq K_{*}d^{\frac{1}{2s}} 
			%\end{align}
			Here, by $\rho $ we mean the distance between $z_{d}$ and $\partial \Omega.$
			\item[(B)] $z_{d} \in \partial \Omega .$\\
			%\item[C.] $z_{d}$ is unique.
		\end{enumerate}
\end{thm}

The plan of the paper is as follows. In Section 2, we recollect known results which are useful for our analysis. In Section 3, we study the regularity of least energy solution of (\ref{P1}) and complete the proof of Theorem \ref{bound1}. In Section 4, we have derived $L^{r}$-estimate for the least energy solutions of (\ref{P1}). Section 5 is devoted to the proof of Theorem \ref{dist}.The proof of inequality (\ref{bound1.3}) (see, next) is a part of Appendix A.
 
\section{Auxiliary Results}
Let us recall the important results which are used in this paper. 
\begin{thm}\label{fembed}(Fractional Sobolev Embedding \cite{Di})  Let $n>2s$ and $2_{s}^{*}=\frac{2n}{n-2s}$ be 
the fractional critical exponent. Then, we have the following inclusions:
\begin{enumerate}
\item for any function $u\in C_{0}(\mathbb{R}^{n})$ and for $q\in [0, 2_{s}^{*}-1]:$ $$ 
\norm{u}_{L^{q+1}(\mathbb{R}^{n})}^{2} \leq B(n,s)\int_{\mathbb{R}^{n}}\int_{\mathbb{R}^{n}}
\frac{\abs{u(x)-u(y)}^{2}}{\abs{x-y}^{n+2s}}dxdy, $$
for some positive constant $B.$ That means $H^{s}(\mathbb{R}^{n})$ is continuously embedded 
in $L^{q+1}(\mathbb{R}^{n}).$
\item Let $\Omega \subset \mathbb{R}^{n}$ be a bounded extension domain for $H^{s}(\Omega).$
Then, the space $H^{s}(\Omega)$ is continuously embedded in $L^{q+1}(\Omega)$ for any $q\in [0, 2_{s}^{*}-1],$ i.e, 
$$ \norm{u}_{L^{q+1}(\Omega)}^{2} \leq B(n,s,\Omega)\norm{u}_{H^{s}(\Omega)}^{2}$$
for some positive constant $B.$
Further, the above embedding is compact for any $q\in [0, 2_{s}^{*}-1).$
\end{enumerate}
\end{thm}
Let $T(\Omega):=\mathbb{R}^{2n}\setminus (\mathbb{R}^{n}\setminus \Omega)^{2}$ be a cross-shaped set on a bounded domain $\Omega \subset \mathbb{R}^n$. Define
\begin{equation}\label{space}
H_{\Omega}^{s}:=\left\{ u: \mathbb{R}^{n}\longrightarrow \mathbb{R}~\text{measurable}: {\norm {u}}_{H_{\Omega}^{s}}< \infty \right\}
\end{equation}
which is equipped with the norm
\begin{equation}
\norm{u}_{H_{\Omega}^{s}}:=\biggl(\norm{u}_{L^{2}
(\Omega)}^{2}+\int_{T(\Omega)}\frac{{\lvert u(x)-u(y)\rvert}^{2}}{{\lvert x-y \rvert}^{n+2s}}dxdy \biggr)^{\frac{1}{2}}.
\end{equation}
\begin{rem}
$H_{\Omega}^{s}$ is a Hilbert space (see \cite{Dip}, Proposition 3.1).
\end{rem}
Let us define the following set:
$$\mathcal{L}_{s}:= \left\{u: \mathbb{R}^{n}\longrightarrow \mathbb{R}~\text{measurable}: \int_{\mathbb{R}^n} \frac{\abs{u(x)}}{1+\abs{x}^{n+2s}}dx < \infty  \right\}.$$ The condition $u \in \mathcal{L}_{s}$ is useful to give a sense to pointwise definition of fractional Laplacian \ref{pointwisedefn}. 
\begin{lem} \label{lessregular}(Lemma 2.3 \cite{Cin})
Let $\Omega \subset \mathbb{R}^n$ be a bounded set. Then $H^{s}_{\Omega} \subset \mathcal{L}_{s}.$
\end{lem}

Next, we recall a few known results about fractional Schr\"{o}dinger equation (\ref{P2}).
%\begin{align}\label{P2}
%(-\Delta)^{s}u+u=\abs{u}^{p-1}u \,\,\, \text{in } \mathbb{R}^{n}.
%\end{align}
\begin{defn}
A measurable function $u: \mathbb{R}^{n} \longrightarrow \mathbb{R}$ is called a weak solution of (\ref{P2}) if it satisfies the following equation
\begin{align*}
	\frac{c_{n,s}}{2} \int_{\mathbb{R}^{n}}\int_{\mathbb{R}^{n}} \frac{(u(x)-u(y))(\psi(x)-\psi(y))}{\abs{x-y}^{n+2s}}dxdy + \int_{\mathbb{R}^{n}}u(x)\psi(x)dx = \int_{\mathbb{R}^{n}} \abs{u(x)}^{p-1}u(x)\psi(x)dx,
\end{align*}
for all $\psi \in C_{0}^{1}(\mathbb{R}^{n}).$
\end{defn}
We define the corresponding energy functional $F: H^{s}(\mathbb{R}^n) \longrightarrow \mathbb{R} $ as follows:
\begin{align}\label{energy2}
F(u):=\frac{1}{2} \Big[ \frac{c_{n,s}}{2} \int_{\mathbb{R}^{n}}\int_{\mathbb{R}^{n}} \frac{\abs{u(x)-u(y)}^{2}}{\abs{x-y}^{n+2s}}dxdy + 
 \int_{\mathbb{R}^{n}}u^{2}dx \Big] -\frac{1}{p+1}\int_{\mathbb{R}^{n}}|u|^{p+1}dx.
\end{align}
The weak solutions of (\ref{P2}) corresponds to the critical points of $F.$ 
\begin{defn}
A function $u \in \mathcal{L}_{s}(\mathbb{R}^{n}) \cap C^{2s+\epsilon}(\mathbb{R}^n),$ when $0<s<\frac{1}{2},$ $2s+\epsilon<1$ or $u \in C^{1, 2s+\epsilon-1}(\mathbb{R}^n) \cap \mathcal{L}_{s}(\mathbb{R}^{n}),$ when $\frac{1}{2}\leq s<1,$ $2s+\epsilon-1<1 $ is said to be a classical solution of (\ref{P2}) if it satisfies the equation (\ref{P2}) pointwise in $\mathbb{R}^{n}.$ 
\end{defn}
Next result gives us a positive, radially symmetric solution of (\ref{P2}), which decays at infinity.
\begin{thm}(Theorem 3.4 \cite{FelQ}) \label{radial5}
Let $u$ be a weak solution of (\ref{P2}). Then $u \in L^{q}(\mathbb{R}^{n}) \cap C^{\alpha}(\mathbb{R}^{n})$ for some $q \in [2, \infty) $ and $\alpha \in (0,1).$ Moreover, 
\begin{align*}
	\lim_{\abs{x} \rightarrow \infty}u(x)=0.
\end{align*}
\end{thm}

\begin{thm}(Theorem 1.3 \cite{FelQ})\label{radial4}
%Let $n \geq 2$ and $p < \frac{n+2s}{n-2s}$ with $ n > 2s, s \in(0, 1).$ 
Equation (\ref{P2}) has a weak solution in $H^{s}(\mathbb{R}^{n}),$ which satisfies $u \geq 0$ a.e. in $\mathbb{R}^{n}.$ Moreover, $u$ is a classical solution that satisfies $u>0$ in $\mathbb{R}^{n}.$
\end{thm}
Following theorem shows that the solutions of (\ref{P2}) has a power type of decay at infinity.
 \begin{thm}(Theorem 1.5 \cite{FelQ})\label{radial2}
 	Let $u$ be a positive classical solution of (\ref{P2}) such that $$\lim_{\abs{x} \rightarrow \infty}u(x)=0.$$ Then, there exist constants $0< C_{1} \leq C_{2}$ such that 
 	\begin{align}
 		\frac{C_{1}}{\abs{x}^{n+2s}}\leq u(x) \leq \frac{C_{2}}{\abs{x}^{n+2s}} \,\,\, \text{for all } \abs{x}\geq 1.
 	\end{align}
 \end{thm}
One can see that there exist some $m>0$ and $s_{0}>0$ such that for $f(u)=u^{p}-u,$ we have
\begin{align}\label{lipschitz}
	\frac{f(v)-f(u)}{v-u}\leq \frac{v^{p}-u^{p}}{v-u}\leq C (v+u)^{m} \text{ for all } 0 < u< v < s_{0}, 
\end{align}
where $C>0$ is some constant. Also, it is simple to see that $f: [0,\infty) \to \mathbb{R}$ is locally Lipschitz. Consequently, we have the following result on radial symmetry and monotonicity property of positive solutions of (\ref{P2}).\underline{}

\begin{thm}(Theorem 1.2 \cite{FelW})\label{radial1}
Let $u$ be a positive classical solution of (\ref{P2}) such that $$\lim_{\abs{x} \rightarrow \infty}u(x)=0.$$ Further, assume that there exists $$t> \max \Bigl\{\frac{2s}{m}, \frac{n}{m+2} \Bigr\}$$ such that $u$ satisfies $u(x)=O\left(\frac{1}{\abs{x}^{t}}\right)$ as $ \abs{x}\rightarrow \infty.$ Then, $u$ is radially symmetric and strictly decreasing about some point in $\mathbb{R}^{n}.$
\end{thm} 
\begin{rem}
Since \begin{align*}
	\frac{C_{1}}{\abs{x}^{n+2s}}\leq u(x) \leq \frac{C_{2}}{\abs{x}^{n+2s}} \,\,\, \text{for all } \abs{x}\geq 1,
\end{align*}  we can take $t=n+2s$ in the above theorem.
\end{rem}
Now, Proposition 4.1\cite{Sech} ascertains  that if $u \in \mathbb{R}^n$ is a weak solution of (\ref{P2}) then $u$ satisfies the following Pohozaev identity:
\begin{align*}
\mathcal{P}(u):=	\frac{(n-2s)c_{n,s}}{4}\int_{\mathbb{R}^n} \int_{\mathbb{R}^n}  \frac{\abs{u(x)-u(y)}^2}{\abs{x-y}^{n+2s}}dxdy + \frac{n}{2}\int_{\mathbb{R}^n}  u^2
dx-\frac{n}{p+1} \int_{\mathbb{R}^n}  u^{p+1}=0.
\end{align*}
Let us define 
\begin{align*}
	\mathcal{G}:=\Bigl\{ u \in H^s(\mathbb{R}^n) \setminus \{0\} \mid \mathcal{P}(u)=0 \Bigr\}.
\end{align*}
In \cite{Chen}, authors have obtained a weak solution $w \in H^{s}(\mathbb{R}^{n})$ of (\ref{P2}) with least energy among all other solutions. In particular, they have proved the following result.
\begin{thm} (Theorem 1.2 \cite{Chen}) \label{radial6}
Equation (\ref{P2}) has a weak solution $w \in H^s(\mathbb{R}^n)$ such that $$0<F(w)=\inf_{u \in \mathcal{G}}F(u) .$$
\end{thm}
\noindent Combining Theorems \ref{radial4}, \ref{radial2}, \ref{radial1} and \ref{radial6} we have the following result.
%We also refer to \cite{Chen, Dip1} for existence and other qualitative properties of a solution to (\ref{P2}). , i.e., $F(u) \leq F(w),$ for any solution $w$ of $(\ref{P2}).$  
\begin{thm}\label{radial3}
	Equation (\ref{P2}) has a positive classical solution $w \in H^{s}(\mathbb{R}^n)$ satisfying
	\begin{enumerate}
		\item[(a)] $w$ has a power type of decay at infinity, i.e., there exist constants $0< C_{1} \leq C_{2}$ such that 
		\begin{align*}
		\frac{C_{1}}{\abs{x}^{n+2s}}\leq w(x) \leq \frac{C_{2}}{\abs{x}^{n+2s}} \,\,\, \text{for all } \abs{x}\geq 1 ;
		\end{align*}
	   \item[(b)] $w$ is radially symmetric, i.e., $w(x)=w(r)$ with $r=\abs{x};$
		\item[(c)] For any non-negative classical solution $u \in  H^{s}(\mathbb{R}^{n})$ of (\ref{P2}), $0<F(w) \leq F(u)$ holds unless $u=0.$
	\end{enumerate}
\end{thm}
%\begin{proof}
%Proof of the statements (a), (b) and (c) follows from Theorems \ref{radial4}, \ref{radial1} and (\ref{radial2}.  For the proof of final statement (d), see Theorem 1.2 in \cite{Chen}.
%\end{proof}
\begin{defn}\label{ground}
We call $w,$ given by Theorem \ref{radial3}, a ground state solution of (\ref{P2}).
\end{defn}

\section{Regularity and bounds for least energy solution $u_{d}$} 
%We consider the following nonlocal  problem  
%\begin{align}\left\{\begin{array}{l l} { d(-\Delta)^{s}u+ u= \abs{u}^{p-1}u } & \text{in $\Omega,$  } \\ 
%\hspace{0.9cm} { \mathcal{N}_{s}u(x)=0 } & \text{in $\mathbb{R}^{n}\setminus \overline{\Omega},$} \\
%\hspace{1.8cm} {u>0} & \text{in $\Omega,$} \end{array} \right.\end{align}
%where $1<p<\frac{n+2s}{n-2s}.$
Let $s\in (0,1)$ and $\Omega \subset \mathbb{R}^{n}$ be a bounded domain of class $C^{1,1}.$
\begin{defn}
A measurable function $u:  \mathbb{R}^{n} \longrightarrow \mathbb{R}$ is said to be a weak solution of (\ref{P1}) if it satisfies the equation
\begin{align}
\frac{dc_{n,s}}{2} \int_{T(\Omega)} \frac{(u(x)-u(y))(\psi (x)-\psi(y))}{\abs{x-y}^{n+2s}}dxdy + 
\int_{\Omega}u(x)\psi(x)dx=\int_{\Omega} \abs{u(x)}^{p-1}u(x)\psi(x)dx, 
\end{align}
for all $\psi \in H_{\Omega}^{s}.$
\end{defn}
We have the following result on the existence of weak solution of (\ref{P1}).
\begin{thm}(Theorem 6.1 \cite{Bar}, Theorem 1.1 \cite{Chen1}) \label{least-soln}
	There exists a nonnegative weak solution $u_{d}$ of  (\ref{P1}) with critical value $c_{d}$, provided $d$ is sufficiently
	small. Moreover, $u_{d}$ satisfies $$0< J_{d}(u_{d}) \leq C d^{\frac{n}{2s}},$$
	where the constant $C$ is independent of $d.$ Consequently, $u_{d}$ is non-constant.
\end{thm}

Define
$$M[v]:=\sup_{t \geq 0}J_{d}(tv), \,\,\, v\in H_{\Omega}^{s}.$$
In the next lemma, we indicate useful characterization of the critical value $c_{d}.$ We follow the similar lines of proof as Lemma 3.1 \cite{Ni1}.

\begin{lem}\label{critical2}
The critical value $c_{d}$ is independent of the choice of $u \in H_{\Omega}^{s}$ such that $u \geq 0,~u \not \equiv 0$ and $J_{d}(u)=0.$ In fact, $c_{d}$ is  the least positive critical value of $J_{d},$ and is given by 
\begin{align}
c_{d}=\inf\biggl\{M[v] \mid v\in H_{\Omega}^{s},~v\not \equiv 0, v\geq 0 \text{ in } \Omega \biggr\}.
\end{align}
\end{lem}
\begin{proof}
For $v\in H_{\Omega}^{s},$ let $$\Omega^{+}=\bigl\{x \in \Omega \mid v(x)>0\bigr\}.$$
Now, for all those $v$ satisfying $\abs{\Omega^{+}}>0,$ define $$g_{d}(t):=J_{d}(tv), \,\,\,\ \text{for } t \geq 0.
$$
First, we will show that $g_{d}(t)$ has a unique maximum. For this, we have 
\begin{align*}
g'_{d}(t)= t\left[ \frac{dc_{n,s}}{2} \int_{T(\Omega)} \frac{\abs{v(x)-v(y)}^{2}}{\abs{x-y}^{n+2s}}dxdy + 
\int_{\Omega}v^{2}dx \right] -t^{p}\int_{\Omega}{v }^{p+1}dx.
\end{align*}
Therefore, $g'_{d}(t_{0})=0$ for some $t_{0}>0$ if and only if 
$$\frac{dc_{n,s}}{2} \int_{T(\Omega)} \frac{\abs{v(x)-v(y)}^{2}}{\abs{x-y}^{n+2s}}dxdy + \int_{\Omega}v^{2}dx = t_{0}^{p-1}\int_{\Omega} v^{p+1}dx.$$
Note that the right hand side is strictly increasing in $t_{0}.$ And hence there exists unique $t_{0}>0$ such that $g'_{d}(t_{0})=0.$ Since $g_{d}(t)>0$ for $t>0$ small and $g_{d}(t) \rightarrow -\infty$ as $t \rightarrow +\infty,$ one easily find that $g_{d}(t)$ has a unique maximum.\\
\indent Let us fix a function $u \not \equiv 0, u\geq 0$ in $H_{\Omega}^{s} $ with $J_{d}(u)=0.$ Let $u_{d}$ be a positive solution of (\ref{P1}) obtained by applying \textit{Mountain-Pass Lemma} and $c_{d}$ the corresponding critical value. We have $J_{d}(u_{d})=c_{d}$ and $J_{d}^{'}(u_{d})=0.$ Since $u_{d}>0$ and $J_{d}^{'}(u_{d})=0,$ we have 
\begin{align}\label{critical1.1}
M[u_{d}]=c_{d},
\end{align}
and hence 
\begin{align}\label{critical1.2}
c_{d}\geq \inf\biggl\{M[v] \mid  v\in H_{\Omega}^{s},~v\not \equiv 0, v\geq 0 \text{ in } \Omega \biggr \}.
\end{align}
On the contrary, assume that the strict inequality occurs in (\ref{critical1.2}). Then, we have $$M[v_{0}]<c_{d},$$ for some $v_{0} \geq 0,\, v_{0} \not \equiv 0 $ in $H_{\Omega}^{s}. $ Therefore, there exists some $t_{1}>0$ such that $t_{1}v_{0}=u_{0}$ satisfies $J_{d}(u_{0})=0.$ Denote by $U$ the subspace of $H_{\Omega}^{s}$ spanned by $u$ and $u_{0}.$ Consider the subset of $U$ defined as follows: $$U^{+}:=\biggl\{ \alpha u + \beta u_{0}\mid \alpha,~\beta \geq 0 \biggr\}.$$ Let $S$ be a circle on $U$ of radius $R$ so large that $R> \max\bigl\{ \norm{u}, \norm{u_{0}}\bigr \}$ and $J_{d}\leq 0$ on $S \cap U^{+}.$ Let $\gamma$ be the path made up of the line segment with endpoints $0$ and 
$\frac{Ru_{0}}{\norm{u_{0}}},$ the circular arc $ S\cap U^{+}$ and the line segment with endpoints $\frac{Ru}{\norm{u}}$ and $u.$ One can easily notice that, along $\gamma$, $J_{d}$ is positive only on the line segment joining $0$ and $u_{0}.$ Hence, we have
 $$ \max_{v \in \gamma}J_{d}(v)=M[v_{0}] < c_{d},$$
 a contradiction to $(\ref{critical1}).$ Thus, we have the equality in $(\ref{critical1.2}),$ i.e.,
 \begin{align}
c_{d}= \inf\biggl\{M[v] \mid v\in H_{\Omega}^{s},~v\not \equiv 0, v\geq 0 \text{ in } \Omega \biggr\}.
\end{align}
Note that $J_{d}(v)=J_{d}(-v)$ for any $v \in H^{s}_{\Omega}$. Since any nontrivial critical point of $J_{d}$ is either positive or negative almost everywhere in $\Omega,$ from the above discussion one can see that $c_{d}$ is the least positive critical value of $J_{d}$. This completes the proof.
\end{proof}
The following lemma gives us the regularity estimate. The similar result is already proved in Lemma 3.6 \cite{Cin}, Remark 4.9 \cite{Cin1}. 
\begin{lem}\label{reg}
	Let $u \in H_{\Omega}^{s}$ be a weak solution of (\ref{P1}). If $u \in L^{\infty}(\Omega)$ then $u \in L^{\infty}(\mathbb{R}^{n}).$ Moreover, 
	\begin{enumerate}
		\item For $0<s<\frac{1}{2},$ $u \in C^{2}(\Omega)$ if $p>3-2s$ and $u \in C^{1, p-2+2s}(\Omega)$ if $2<p\leq 3-2s.$
		\item For $\frac{1}{2} \leq s< 1,$ $u \in C^{2}(\Omega).$ 
	\end{enumerate}
\end{lem}

%In order to study the shape of solution on domain $\Omega$, it is necessary to know about the bound of solution $u_{d}.$ 
   Now, we prove that the least energy solution $u_{d}$ is bounded by some constant independent of $d.$ \\

\noindent\textbf{Proof of Theorem \ref{bound1}.} The proof of the first inequality of Theorem \ref{bound1} is fairly standard and simple, which can be seen in the literature, for instance, see Theorem 1.1 \cite{Chen}. 
Since it is short, for the sake of completeness, we include it here.  For this, we have
\begin{align}
	J_{d}(u_{d}):=\frac{1}{2} \left[ \frac{c_{n,s}d}{2} \int_{T(\Omega)} \frac{\abs{u_{d}(x)-u_{d}(y)}^{2}}{\abs{x-y}^{n+2s}}dxdy + 
	\int_{\Omega}u^{2}dx \right] -\frac{1}{p+1}\int_{\Omega}u_{d}^{p+1}dx.
\end{align}
Since $u_{d}$ is a critical point of $J_{d}$, we have
\begin{align}
	J_{d}{'}(u_{d})=0 \text{ on } H_{\Omega}^{s} .
\end{align} This implies that
\begin{align}\label{bound1.18}
	d\frac{c_{n,s}}{2} \int_{T(\Omega)} \frac{\abs{u_{d}(x)-u_{d}(y)}^{2}}{\abs{x-y}^{n+2s}}dxdy + 
	\int_{\Omega}u_{d}^{2}dx  =\int_{\Omega}u_{d}^{p+1}dx.
\end{align}
Hence from above equations, we get
\begin{align}\label{bound1.19}
	J_{d}(u_{d})=& \, \left(\frac{1}{2}-\frac{1}{p+1}\right)\int_{\Omega}u_{d}^{p+1}dx \\
	=& \, \frac{(p-1)}{2(p+1)} \int_{\Omega}u_{d}^{p+1}dx.
\end{align}
Now, by Theorem \ref{least-soln}, we have $J_{d}(u_{d})\leq C d^{\frac{n}{2s}}$, where the constant $C$ depends only on $p.$
Using this inequality in the above equation, we get
\begin{align*}
	\int_{\Omega}u_{d}^{p+1}dx  \leq \frac{2(p+1)}{p-1}Cd^{\frac{n}{2s}}.
\end{align*}
Taking $C_{0}=\frac{2(p+1)}{p-1}C$, proves the first inequality of Theorem \ref{bound1}.\\ The proof of second inequality of Theorem \ref{bound1} is little constructive. We claim that $$\displaystyle \sup_{\Omega}u_{d}(x) \leq C_{1}$$ for some constant $C_{1}>0$ depending on $p$ and $\Omega$ only. 
Multiplying $(\ref{P1})$ by $u_{d}^{2t-1}$ and integrating over $\Omega,$ we get 
\begin{align}\label{bound1.31}
	\frac{c_{n,s}d}{2}\int_{T(\Omega)} \frac{(u_{d}(x)-u_{d}(y))(u_{d}^{2t-1}(x)-u_{d}^{2t-1}(y))}{\abs{x-y}^{n+2s}}dxdy + \int_{\Omega}u_{d}^{2t}dx  = \int_{\Omega}u_{d}^{p+2t-1}dx.
\end{align}
Now, we use the following inequality. We have given the proof of this inequality in appendix. Let $x,y \geq 0$ are real numbers and $k\geq 1$, then we have 
\begin{align}\label{bound1.3}
	\frac{1}{k}(x^{k}-y^{k})^{2}\leq (x-y)(x^{2k-1}-y^{2k-1}).
\end{align}
Consequently, we have 
\begin{align}\label{bound1.32}
	\frac{1}{t}\int_{T(\Omega)} \frac{(u_{d}^{t}(x)-u_{d}^{t}(y))^{2}}{\abs{x-y}^{n+2s}}dxdy \leq \int_{T(\Omega)} \frac{(u_{d}(x)-u_{d}(y))(u_{d}^{2t-1}(x)-u_{d}^{2t-1}(y))}{\abs{x-y}^{n+2s}}dxdy.
\end{align}
From  (\ref{bound1.31}) and (\ref{bound1.32}), we get
\begin{align}\label{bound1.4}
	\frac{dc_{n,s}}{2t}\int_{T(\Omega)} \frac{(u_{d}^{t}(x)-u_{d}^{t}(y))^{2}}{\abs{x-y}^{n+2s}}dxdy + \int_{\Omega}u_{d}^{2t}dx  \leq \int_{\Omega}u_{d}^{p+2t-1}dx.
\end{align}
Now, by the fractional Sobolev embedding Theorem \ref{fembed},
\begin{align}\label{bound1.5}
	\Bigl(\int_{\Omega}{\abs{v}^{2_{s}^{*}}} \Big)^{2/2_{s}^{*}} \leq \frac{A}{d}\Bigl(d\frac{c_{n,s}}{2}\int_{\Omega}\int_{\Omega} \frac{\abs{v(x)-v(y)}^{2}}{\abs{x-y}^{n+2s}}dxdy + \int_{\Omega}\abs{v}^{2}dx \Bigr),
\end{align}
where $d\in (0,d_{0})$ for some $d_{0}>0$, $A>0$ some constant, $v\in H^{s}(\Omega),$ and $2_{s}^{*}=\frac{2n}{n-2s}.$
\noindent The embedding constant $A$ depends only on $n,$ $s,$ $d_{0},$ and $\Omega.$
To see this, let us define \begin{align*}
	\Omega_{d}:=\Bigl \{ y : \frac{y}{d^{1/2s}} \in \Omega \Bigr \} \text{ and } w(y):=v \Bigl (\frac{y}{d^{1/2s}} \Bigr ),\text{ where } y \in \Omega_{d}.
\end{align*}
Now, we have
\begin{align}
	d\int_{\Omega}\int_{\Omega} \frac{\abs{v(x)-v(y)}^{2}}{\abs{x-y}^{n+2s}}dxdy + \int_{\Omega}{v}^{2}dx=& \, \frac{1}{d^{\frac{n}{2s}}} \Biggl [ \int_{\Omega_{d}}\int_{\Omega_{d}} \frac{\abs{v(\Bigl (\frac{x'}{d^{\frac{1}{2s}}} \Bigr )-v(\Bigl (\frac{y'}{d^{\frac{1}{2s}}} \Bigr )}^{2}}{\abs{x'-y'}^{n+2s}}dx'dy' + \int_{\Omega_{d}}{v\Bigl (\frac{x'}{d^{\frac{1}{2s}}} \Bigr )}^{2}dx' \Biggr ] \\
	=& \, \frac{1}{d^{\frac{n}{2s}}} \Biggl [ \int_{\Omega_{d}}\int_{\Omega_{d}} \frac{\abs{w(x')-w(y')}^{2}}{\abs{x'-y'}^{n+2s}}dx'dy' + \int_{\Omega_{d}}{w(x')}^{2}dx' \Biggr ] \\
	\geq & \, \frac{A}{d^{\frac{n}{2s}}} \Bigl(\int_{\Omega_{d}}{\abs{w}^{2_{s}^{*}}} dx' \Big)^{\frac{2}{2_{s}^{*}}} \\
	=& \, Ad^{\bigl(\frac{2}{2_{s}^{*}}-1 \bigr)\frac{n}{2s}} \Bigl(\int_{\Omega}{\abs{v}^{2_{s}^{*}}} dx \Big)^{\frac{2}{2_{s}^{*}}} .
\end{align}
Therefore, we observe that $A$ is uniform for $d \in (0, d_{0}).$\\
Note that $\Omega \times \Omega \subset T(\Omega).$ Then by virtue of (\ref{bound1.4}) and (\ref{bound1.5}), we have
\begin{align}\label{bound1.6}
	\Bigl(\int_{\Omega}{\abs{u_{d}}^{t2_{s}^{*}}} \Big)^{\frac{2}{2_{s}^{*}}} \leq \frac{tA}{d}\int_{\Omega}u_{d}^{p+2t-1}dx.
\end{align}
Now, we define two sequences $\bigl\{L_{j} \bigr \}$ and $\bigl\{M_{j} \bigr \}$ by the following recurrence relations: \\ 
\begin{align}\label{bound1.7}
	p-1+2L_{0}= & \, 2_{s}^{*}, \nonumber \\
	p-1+2L_{j+1}= & \, 2_{s}^{*}L_{j}, \,\,\,\, j=0,1,2,...
\end{align}
\begin{align}\label{bound1.8}
	M_{0}=& \, (AC_{0})^{\frac{2_{s}^{*}}{2}}, \nonumber \\
	M_{j+1}=& \, (AL_{j}M_{j})^{\frac{2_{s}^{*}}{2}}, \,\,\,\, j=0,1,2,...
\end{align}
We note that $L_{j}$ is explicitly given by 

\begin{align}\label{bound1.9}
	L_{j}=\frac{1}{(2_{s}^{*}-2)}\left( \Bigl(\frac{2_{s}^{*}}{2}\Bigr)^{j+1}(2_{s}^{*}-p-1)+p-1 \right).
\end{align}
Since $1<p<2_{s}^{*}-1$, it follows that $L_{j}\geq 1$ for all $j \geq 0$ and $L_{j} \rightarrow \infty$ as $j \rightarrow \infty.$
We shall show that 
\begin{align}\label{bound1.10}
	\int_{\Omega}u_{d}^{p-1+2L_{j}}dx \leq & \, M_{j}d^{\frac{n}{2s}} \,\,\,\,\text{for all } j\geq 0, 
\end{align}	
and
\begin{align}	\label{bound1.114} 
M_{j}\leq e^{mL_{j-1}}
\end{align}
for some constant $m>0.$ Then, we have 
$$ \sup_{\Omega}u_{d}(x)\leq C_{1},$$
where $C_{1}>0$ depending only on $C_{0}$ and $\Omega.$
In fact (\ref{bound1.9}) and (\ref{bound1.10}) entail us 
\begin{align}\label{bound1.15}
	\norm{u}_{L^{2_{s}^{*}L_{j-1}}(\Omega)} \leq & \, \Bigl( e^{mL_{j-1}}d^{\frac{n}{2s}}\Bigr )^{\frac{1}{(2_{s}^{*}L_{j-1})}} \nonumber \\
	= & \, e^{\frac{m}{2_{s}^{*}} d^{\frac{(n-2s)}{4L_{j-1}}}}
\end{align}
and hence letting $j \rightarrow \infty,$ we obtain 
$$\norm{u}_{L^{\infty}(\Omega)} \leq e^{\frac{m}{2_{s}^{*}}}.$$

First, we verify (\ref{bound1.10}). By virtue of (\ref{bound1.1}) and (\ref{bound1.5}), we have
\begin{align}
	\Bigl(\int_{\Omega}{\abs{u_{d}}^{2_{s}^{*}}} \Big)^{\frac{2}{2_{s}^{*}}} \leq & \, \frac{A}{d}\Bigl( \frac{c_{n,s}d}{2}\int_{T(\Omega)} \frac{\abs{u_{d}(x)-u_{d}(y)}^{2}}{\abs{x-y}^{n+2s}}dxdy + \int_{\Omega}\abs{u_{d}}^{2}dx \Bigr) \nonumber \\
	\leq & \, \frac{A}{d}C_{0}d^{\frac{n}{2s}} \nonumber \\
	=& \, AC_{0}d^{\frac{n}{s2_{s}^{*}}}.
\end{align}
Hence, (\ref{bound1.10}) holds for $j=0.$ Suppose that we have proved (\ref{bound1.10}) for $j \geq 0.$ Then by (\ref{bound1.6}), we have
\begin{align}
	\int_{\Omega}{\abs{u_{d}}^{p-1+2L_{j+1}}}dx \leq & \, \Bigl (\frac{L_{j}A}{d}\int_{\Omega}u_{d}^{p+2L_{j}-1}dx \Bigr )^{\frac{2_{s}^{*}}{2}} \nonumber \\
	\leq & \, \Bigl ( AL_{j}d^{-1}M_{j}d^{\frac{n}{2s}}\Bigr )^{\frac{2_{s}^{*}}{2}} \nonumber \\
	=& \, \Bigl(AL_{j}M_{j} \Bigr)^{\frac{2_{s}^{*}}{2}}d^{\frac{n}{2s}}.
\end{align}
This implies that (\ref{bound1.10}) is also true for $j+1.$ Therefore it remains to show (\ref{bound1.114}). Put 
\begin{align} \label{bound1.37}
	\lambda_{j}= & \, \frac{2_{s}^{*}}{2} \cdot \log(AL_{j}) \text{ and } \eta_{j}=  \log(M_{j}).
\end{align}
Hence
\begin{align}
\eta_{j+1}=& \, \frac{2_{s}^{*}}{2}\cdot \eta_{j} +\lambda_{j} .
\end{align}
The explicit value of $L_{j}$ is given by 
\begin{align}\label{bound1.13}
	L_{j}= (2_{s}^{*}-2)^{-1}\Bigl ( (2^{-1}2_{s}^{*})^{j+1}(2_{s}^{*}-p-1)+p-1 \Bigr ).
\end{align}
Now, we have
\begin{align}\label{bound1.11}
	\lambda_{j}=& \, \frac{2_{s}^{*}}{2}\log\Bigl[ \frac{A}{(2_{s}^{*}-2)}\Bigl((2^{-1}2_{s}^{*})^{j+1}(2_{s}^{*}-p-1)+p-1 \Bigr ) \Bigr ]\\
	=& \, \frac{2_{s}^{*}}{2} \Bigl [ \log(A(2_{s}^{*}-2))+ \log \bigl ((2^{-1}2_{s}^{*})^{j+1}(2_{s}^{*}-p-1)+p-1 \bigr ) \Bigr ].
\end{align}
Therefore, we can find some $C^{*}$ such that 
\begin{align} \label{bound1.12}
	\lambda_{j} \leq C^{*}(j+1).
\end{align}
We now define a sequence $\bigl \{\gamma_{j} \bigr \}$ by 
\begin{align}\label{bound1.38}
	\gamma_{0}= \eta_{0} \text{ and } \gamma_{j+1}= \frac{2_{s}^{*}}{2}\gamma_{j} + C^{*}(j+1)
\end{align}
for $j \geq 1.$ Clearly, $\eta_{j}  \leq \gamma_{j}$ for all $j \geq 0.$ Moreover, since 
\begin{align*}
	\gamma_{j}= \Bigl (\frac{2_{s}^{*}}{2} \Bigr )^{j}\bigl( \eta_{0}+2C^{*}2_{s}^{*}(2_{s}^{*}-2)^{-2} \bigl )-2C^{*}(2_{s}^{*}-2)^{-1}\bigl (j+ 2_{s}^{*}(2_{s}^{*}-2)\bigr),
\end{align*}
 in view of (\ref{bound1.13}), there exists $m>0$ such that $\gamma_{j} \leq mL_{j-1}.$ Hence $\log(M_{j}) \leq m L_{j-1}$ and we obtain (\ref{bound1.114}).
Note that $m$ depends only on $\eta_{0},$ $2_{s}^{*}$ and $C^{*};$ whereas $C^{*}$ depends only on $2_{s}^{*},$ $p$ and $A.$ This completes the proof. \qed \\

\begin{rem}
	It is known that if $u \in \mathcal{L}_{s}(\mathbb{R}^{n})\cap C^{2s+\epsilon}(\Omega), $ when $0<s<\frac{1}{2}, 2s+\epsilon<1$ or $u \in \mathcal{L}_{s}(\mathbb{R}^{n})\cap C^{1,2s+\epsilon-1}(\Omega), $ when $\frac{1}{2}\leq s<1, 2s+\epsilon-1<1$, one can compute $(-\Delta)^{s}u(x)$ pointwise for all $x$ in $\Omega.$ In fact, one can write
	\begin{align*}
		(-\Delta)^{s}u(x)&=c_{n,s}P.V. \int_{\mathbb{R}^{n}} \frac{u(x)-u(y)}{\abs{x-y}^{n+2s}}dy \\
	\end{align*}
\end{rem}

\begin{defn}\label{classical2}
	We say that $u: \mathbb{R}^{n} \longrightarrow \mathbb{R}$ is a classical solution of (\ref{P1}) if it satisfies the following:
	\begin{enumerate}
		\item $u \in \mathcal{L}_{s}(\mathbb{R}^{n})\cap C^{2s+\epsilon}(\Omega), $ when $0<s<\frac{1}{2}, 2s+\epsilon<1$ or $u \in \mathcal{L}_{s}(\mathbb{R}^{n})\cap C^{1,2s+\epsilon-1}(\Omega), $ when $\frac{1}{2}\leq s<1, 2s+\epsilon-1<1.$
		\item $\mathcal{N}_{s}u(x)=0, \,\,\, x \in \mathbb{R}^{n}\setminus \Omega ;$
		\item $d(-\Delta)^{s}u(x)+u(x)= \abs{u(x)}^{p-1}u(x)$ pointwise for all $x \in \Omega.$
	\end{enumerate}
\end{defn}

We make similar remarks as in \cite{Bia}, which offers a relation between the weak and classical solutions of $(\ref{P1}).$

\begin{rem}
	Let $u_{d}$ be a least energy solution of (\ref{P1}) in $H^{s}_{\Omega}.$ Then by Lemma \ref{lessregular}, Theorem \ref{bound1} and Lemma \ref{reg}, we have 
		\begin{enumerate}
		\item for $0<s<\frac{1}{2},$ $u_{d} \in \mathcal{L}_{s}(\mathbb{R}^n)\cap C^{2}(\Omega)$ if $p>3-2s$ and $u_{d} \in \mathcal{L}_{s}(\mathbb{R}^n) \cap C^{1, p-2+2s}(\Omega)$ if $2<p\leq 3-2s;$
		\item for $\frac{1}{2} \leq s < 1,$ $u_d  \in \mathcal{L}_{s}(\mathbb{R}^n) \cap C^{2}(\Omega).$ 
	\end{enumerate}
	
	Now, using nonlocal integration by parts formulae given in \cite{Dip}, one can easily check that $$d(-\Delta)^{s}u_{d}(x) + u_{d}(x)=\abs{u_{d}(x)}^{p-1}u_{d}(x)$$ holds pointwise in $\Omega.$ This implies that $u_{d}$ is a classical solution of (\ref{P1}). Conversely, if $u_{d}$ is a classical solution of (\ref{P1}) satisfying $u_{d} \in H^{s}_{\Omega},$ then $u_{d}$ is a weak solution of  (\ref{P1}).
\end{rem}
%\begin{thm}\label{bound1}
%Let $u_{d}$ be a weak solution to $\ref{P1}.$ Then 
%\begin{align}\label{bound1.1}
%\frac{c_{n,s}}{2} \int_{T(\Omega)} \frac{\abs{u_{d}(x)-u_{d}(y)}^{2}}{\abs{x-y}^{n+2s}}dxdy + 
% \int_{\Omega}u^{2}dx=& \int_{\Omega} u_{d}^{p+1}dx \leq C_{0}d^{\frac{n}{2s}},
%\end{align}
%where $C_{0}>0$ is some constant depending on $p.$
%Moreover, there is a constant $C_{1}>0$ depending on $C_{0},$ and $\Omega$ such that \begin{align}\label{bound1.2}
%\sup_{\Omega}u_{d}(x)\leq C_{1}
%\end{align}
%\end{thm}
The following lemma shows that the maximum of least energy solution is always greater than unity. 
\begin{lem}\label{sup1}
Let $u_d$ be the least energy solution of (\ref{P1}). Let \begin{align}
M_{d}=\sup_{x\in \overline{\Omega}} u_{d}(x).
\end{align}
Then $M_{d}>1.$
\end{lem}
\begin{proof}
Since $u_{d}$ is a weak solution of (\ref{P1}), we get 
\begin{align}
d\frac{c_{n,s}}{2}\int_{T(\Omega)}\frac{(u_{d}(x)-u_{d}(y))(w(x)-w(y))}{\abs{x-y}^{n+2s}}dxdy +  \int_{\Omega}u_{d}w dx
=\int_{\Omega} u_{d}^{p}w dx \,\,\,\,\text{holds,}\,\,\,\forall\, w \in H_{\Omega}^{s}.
\end{align}
Taking $w=1$ in the above equation, we get
\begin{align*}
\int_{\Omega}u_{d}(x)dx=& \,\int_{\Omega}u_{d}^{p}(x)dx.
\end{align*}
This implies that
\begin{align*}
\int_{\Omega}u_{d}(x)(1-u_{d}^{p-1}(x))dx=& \,0.
\end{align*}
Now, if 
$u_{d}(x)\leq 1,$ for all $x \in \overline{\Omega},$ then \\
$$1-u_{d}(x)\geq 0, \forall x \in {\overline{\Omega}} .$$

Thus from the above equation, we get that $u_{d}(x)= 1~a.e.$ in $\overline{\Omega}.$  Now, by Lemma \ref{reg}, we can assume that $u_{d}$ is continuous and hence  $u_{d}\equiv 1$ in $\overline{\Omega},$ a contradiction to our assumption that $u_{d}$ is a non-constant solution. Therefore, there exists $x_{0}$ in $\overline{\Omega}$ such that $u_{d}(x_{0})>1.$ Thus $M_{d}>1.$ 
\end{proof}

\section{$L^{r}$- estimates on $u_{d}$}
Here, we derive $L^{r}$-estimate for $u_{d}.$ Following results are generalization to the nonlocal case of 
Proposition 2.2 and Lemma 2.3 \cite{Lin1}.
\begin{prop}\label{bound2}
	For $d_{0}>0$ fixed, there is a constant $K_{0}$ such that 
	\begin{align}\label{bound2.1}
		d\frac{c_{n,s}}{2}\int_{T(\Omega)}\frac{(u_{d}(x)-u_{d}(y))^{2}}{\abs{x-y}^{n+2s}}dxdy +  \int_{\Omega}u_{d}^{2} dx \geq K_{0}d^{\frac{n}{2s}},
	\end{align}
	where $u_{d}$ is the least energy solution of (\ref{P1}) with $0<d<d_{0}.$
\end{prop}
\begin{proof}
	On contrary, suppose that there is a sequence $\bigl\{d_{k}\bigr\}$ contained in the interval $(0,d_{0})$ and a sequence of positive solutions $\bigl\{u_{k}\bigr\}$ to (\ref{P1}) with $d=d_{k}$ such that 
	\begin{align}
		\zeta_{k}:= \frac{1}{d^{\frac{n}{2s}}}\left(d \frac{c_{n,s}}{2} \int_{T(\Omega)}\frac{(u_{k}(x)-u_{k}(y))^{2}}{\abs{x-y}^{n+2s}}dxdy +  \int_{\Omega}u_{k}^{2} dx \right) \rightarrow 0 \text{ as } k\rightarrow \infty.
	\end{align}
	We are going to follow the same arguments as used in the proof of Lemma \ref{bound1} to prove this proposition. Once again define the sequences $\bigl\{L_{k}\bigr\}$ and $\bigl\{M_{j}\bigr\}$ as defined earlier in (\ref{bound1.7}) and (\ref{bound1.8}), respectively. Instead of $C_{0},$ we write $\zeta
	_{k}$ in the definition of $\bigl\{M_{j}\bigr\}$: 
	\begin{align}
		p-1+2L_{0}= & \, 2_{s}^{*}, \nonumber \\
		p-1+2L_{j+1}= & \, 2_{s}^{*}L_{j}, \,\,\,\, j=0,1,2,\dots
	\end{align}
	and
	\begin{align}
		M_{0}=& \, (A\zeta_{k})^{\frac{2_{s}^{*}}{2}}, \nonumber \\
		M_{j+1}=& \, (AL_{j}M_{j})^{\frac{2_{s}^{*}}{2}}, \,\,\,\, j=0,1,2,\dots
	\end{align}
	Further, define the sequences $\bigl\{\lambda_{j}\bigr\},$ $\bigl\{\eta_{j}\bigr\},$ and $\bigl\{\gamma_{j}\bigr\}$ as defined earlier in (\ref{bound1.37}) and (\ref{bound1.38}). From (\ref{bound1.10}), we have
	\begin{align}\label{bound2.2}
		\left(\int_{\Omega}u_{k}^{2_{s}^{*}L_{j-1}}dx\right)^{(2_{s}^{*}L_{j-1})} \leq & \, \left(M_{j}d_{k}^{n/2s} \right)^{1/(2_{s}^{*}L_{j-1})}.
	\end{align}
	Since $$\log(M_{j})=\eta_{j} \leq \gamma_{j},$$ we have 
	\begin{align}
		\frac{\log\left(M_{j}\right)}{2_{s}^{*}L_{j-1}} \leq \frac{\eta_{j}}{2_{s}^{*}L_{j-1}}.
	\end{align}
	Now,
	\begin{align*}
		\lim_{j \rightarrow \infty} \frac{\eta_{j}}{2_{s}^{*}L_{j-1}}= & \, \lim_{j \rightarrow \infty}\frac{\Bigl (\frac{2_{s}^{*}}{2} \Bigr )^{j}\Bigl[ \eta_{0}+2C^{*}2_{s}^{*}(2_{s}^{*}-2)^{-2} \Bigr]-2C^{*}(2_{s}^{*}-2)^{-1}\Bigl [j+ 2_{s}^{*}(2_{s}^{*}-2)\Bigr]}{\frac{2_{s}^{*}}{(2_{s}^{*}-2)}\Bigl [ \Bigl(\frac{2_{s}^{*}}{2}\Bigr)^{j}(2_{s}^{*}-p-1)+p-1 \Bigr ]}\\
		=& \, \frac{(2_{s}^{*}-2)(\eta_{0}+2C^{*}2_{s}^{*}(2_{s}^{*}-2)^{-2})}{2_{s}^{*}(2_{s}^{*}-p-1)}.
	\end{align*}
	Letting $j \rightarrow \infty$ in (\ref{bound2.2}), we get
	\begin{align}\label{bound2.3}
		\norm{u_{k}}_{L^{\infty}(\Omega)} \leq e^{a_{1}(\eta_{0}+a_{2})},
	\end{align} 
	with $a_{1}$ and $a_{2}$ depending only on $2_{s}^{*},~p$ and $C^{*}.$
	Since
	\begin{align*}
		\eta_{0}= \log (M_{0})= \frac{2_{s}^{*}}{2}\log(A\zeta_{k}).
	\end{align*}
	Therefore, as $k \rightarrow \infty,~\eta_{0}\rightarrow -\infty.$ Thus, in view of (\ref{bound2.3}), we get
	\begin{align*}
		\norm{u_{k}}_{L^{\infty}(\Omega)} \rightarrow 0,
	\end{align*}
	which leads to a contradiction to Lemma \ref{sup1}.
\end{proof}

\noindent\textbf{Proof of the Theorem \ref{lr-estimate}:}
	First, we will show the second part  of Inequality (\ref{lr1.1}).\\
	
	\noindent \textbf{Case-I.} $r \geq 2_{s}^{*}=\frac{2n}{n-2s}.$ \\
	Let $\bigl\{L_{j} \bigr \}$ be the sequence defined in (\ref{bound1.7}). If $r \in \bigl\{2_{s}^{*}L_{j} \bigr \}$, then the second inequality of (\ref{lr1.1}) follows from (\ref{bound1.10}). So assume that $2_{s}^{*}L_{j}<r< 2_{s}^{*}L_{j+1}$ for some $j \geq 0.$ We have $$r=t2_{s}^{*}L_{j}+(1-t)2_{s}^{*}L_{j+1}, \text{ for some } t \in (0,1).$$
	Using H\"{o}lders inequality and (\ref{bound1.10}), we get
	\begin{align*}
		\int_{\Omega}u_{d}^{r}dx = & \, \int_{\Omega}u_{d}^{t2_{s}^{*}L_{j}+(1-t)2_{s}^{*}L_{j+1}}dx, \\
		\leq & \, \left( \int_{\Omega}u_{d}^{2_{s}^{*}L_{j}}dx \right)^{t}\left( \int_{\Omega}u_{d}^{2_{s}^{*}L_{j+1}}dx \right)^{1-t} \\
		\leq & \, \left(M_{j-1}d^{n/2s} \right)^{t}(M_{j}d^{n/2s})^{1-t}\\
		= & \,  M_{j-1}^{t}M_{j}^{1-t}d^{\frac{n}{2s}}.
	\end{align*}
	\textbf{Case-II.} $2 \leq r \leq 2_{s}^{*}$.\\
	We write $$r=2t+(1-t)2_{s}^{*},$$ for some $t \in [0,1].$ Then, using H\"{o}lder's inequality, from Equations (\ref{bound1.1}) and (\ref{bound1.10}) with $j=0$, we get
	\begin{align*}
		\int_{\Omega}u_{d}^{r}dx \leq & \, \left( \int_{\Omega}u_{d}^{2}dx \right)^{t} \left( \int_{\Omega}u_{d}^{2_{s}^{*}}dx \right)^{1-t}\\
		\leq & \, C_{0}^{t}M_{0}^{(1-t)}d^{\frac{n}{2s}},
	\end{align*}
	where the constant $C_{0}$ is independent of $d.$
	
	\noindent \textbf{Case-III.} $1 \leq r < p+1.$ \\
	Integrating both sides of $(\ref{P1})$ and using the condition $\mathcal{N}_{s}u(x)=0$, for $x\in \mathcal{C}\Omega ,$ we get 
	\begin{align}\label{lr1.2}
		\int_{\Omega}u_{d}dx= \int_{\Omega}u_{d}^{p}dx.
	\end{align}
	It is easy to see that $$p=t+(1-t)(p+1) \text{ with } t=\frac{1}{p} \in (0,1).$$ Notice that $p+1 \in (2, 2_{s}^{*}).$ Therefore, using the H\"{o}lder's inequality and (\ref{lr1.2}), we get 
	\begin{align*}
		\int_{\Omega}u_{d}^{p}dx \leq & \, \left( \int_{\Omega}u_{d}dx \right)^{t}\left( \int_{\Omega}u_{d}^{p+1}dx \right)^{(1-t)}, \\
		\int_{\Omega}u_{d}^{p}dx \leq & \, \int_{\Omega}u_{d}^{p+1}dx \leq C_{0}d^{\frac{n}{2s}} \,\,\,\, \left(\text{by \eqref{bound1.1}}\right),
	\end{align*}
	where the constant $C_{0}$ depends only upon $p+1.$ \\
	Also, in the view of (\ref{lr1.2}) and (\ref{bound1.1}), we observe that the second inequality of (\ref{lr1.1}) holds for $r=1.$ Now, repeating the interpolation between $1$ and $p+1$, we see that the second inequality of (\ref{lr1.1}) holds for all r $\geq 1.$ \\
	
	\noindent \textbf{Case-IV.} Let $0<r \leq 1.$
	Taking $F=u_{d}^{r},\, G=1, \, p=\frac{1}{r}, \, q=\frac{1}{1-r}$ and using the H\"{o}lders inequality, we get \
	\begin{align*}
		\int_{\Omega}u_{d}^{r}dx \leq \norm{F}_{p}\norm{G}_{q}=\abs{\Omega}^{1-r} \left(\int_{\Omega}u_{d}dx \right)^{r} \leq \abs{\Omega}^{1-r} B(1)^{r}d^{\frac{nr}{2s}}.
	\end{align*}
	This proves the second inequality of (\ref{lr1.4}).\\
	Now, let us prove the first inequality of (\ref{lr1.1}) and (\ref{lr1.4}).
	In view of equations (\ref{bound1.18}) and (\ref{bound2.1}), we see that
	\begin{align}\label{lr2.3}
		\int_{\Omega}u_{d}^{p+1} \geq K_{0}d^{\frac{n}{2s}}.
	\end{align}
	Since $$\displaystyle \sup_{\Omega}u_{d}(x) \leq C_{1}, \text{ for some constant } C_{1}>0,$$ we have
	\begin{align*}
		K_{0}d^{\frac{n}{2s}}\leq \int_{\Omega}u_{d}^{p+1} &=  \, \int_{\Omega}\left(u_{d}^{p+1-r}\right)\left(u_{d}^{r}\right)dx \\
		&\leq C_{1}^{p+1-r}\int_{\Omega}u_{d}^{r}dx.
	\end{align*}
	This implies that $$\int_{\Omega}u_{d}^{r}dx \geq K_{0}C_{1}^{r-p-1}d^{\frac{n}{2s}}, ~r<p+1.$$ 
	For $r>p+1$, we write $p+1=1+(1-t)r.$
	Therefore, we get 
	\begin{align*}
		K_{0}d^{\frac{n}{2s}}\leq & \, \int_{\Omega}u_{d}^{p+1}dx \\
		=& \, \int_{\Omega}u_{d}^{1+(1-t)r}dx \\
		\leq & \, \Bigl(u_{d}dx \Bigr)^{t}\Bigl(u_{d}^{r}dx \Bigr)^{1-t} \\
		\leq & \, \Bigl( B(1)d^{\frac{n}{2s}} \Bigr)^{t} \Bigl( u_{d}^{r}dx \Bigr)^{1-t}.
	\end{align*}
	This yields that $$\int_{\Omega}u_{d}^{r}dx \geq (K_{0}B(1)^{-t})^{\frac{1}{1-t}}d^{\frac{n}{2s}}.$$ \qed

%\begin{thm}\label{dist}
%Let $u_{d}$ be a nonconstant least energy solution to (\ref{P1}). If $u_{d}$ achieves maximum at a point $z_{d} \in \bar{\Omega},$ then for all $d$ sufficiently small we have the following:
%\begin{enumerate}
%\item[A.]There exists a positive constant $K_{*}$ such that $\rho (z_{d}, \partial \Omega) \leq K_{*}d^{\frac{1}{2s}} .$
%%\begin{align}\label{dist1.1}
%%\rho (z_{d}, \partial \Omega) \leq K_{*}d^{\frac{1}{2s}} 
%%\end{align}
%Here, by $\rho $ we mean the distance metric.
%\item[B.] $z_{d} \in \partial \Omega .$
%\item[C.] $z_{d}$ is unique.
%\end{enumerate}
%\end{thm}
%In simple words, the above theorem tells that a least energy solution to (\ref{P1}) attains maximum at unique point on the boundary of domain $\Omega.$  \\
\section{{Proof of theorem \ref{dist}}}
We prove Theorem \ref{dist} in this section. Its proof is more involved and requires some scaling and compactness arguments. We prove the statements of theorem one by one. Let $z_{d} \in \overline{\Omega} $ be a point of maximum of $u_{d}.$ The basic idea for its proof is simple. We approximate $u_{d}$ around $z_{d}$ by a scaled positive radial solution of (\ref{P2}). It gives us an upper bound on $c_{d},$ which is closely related to the location of point $z_{d}.$ \\
%\noindent \textbf{Proof of Theorem \ref{dist}}.

\noindent \textbf{Proof of (A).} If the inequality in (A) is not true, then there is a decreasing sequence $d_{j}\downarrow 0$ such that
\begin{align}
\rho_{j}:=\frac{\rho(z_{j}, \partial \Omega)}{d_{j}^{\frac{1}{2s}}} \rightarrow +\infty \text{ as $j \rightarrow \infty$},
\end{align}
where $z_{j}:=z_{d_{j}}$ is a point of maximum of $u_{d_{j}}$ on $\overline{\Omega}.$ Define $$\phi_{j}(y):=u_{d_{j}}(yd_{j}^{\frac{1}{2s}}+z_{j}) \,\,\, \text{for } y \in \mathbb{R}^{n}.$$

 Since $u_{d}$ is a classical solution of (\ref{P1}), we have
\begin{align}
	(-\Delta)^{s}\phi_{j}+\phi_{j}=\phi_{j}^{p} \,\,\, \text{in } B_{\rho_{j}},
\end{align}
and  \begin{enumerate}
\item $ \phi_{j} \in C^{0,2s+\epsilon}(B_{\rho_{j}}),$ when $0<s<\frac{1}{2},$ $ 2s+\epsilon<1 $
\item  $ \phi_{j} \in C^{1,2s+\epsilon-1}(B_{\rho_{j}}),$ when $\frac{1}{2} \leq s <1,$ $2s+\epsilon-1<1.$
\end{enumerate}
First, we claim that the sequence $\bigl\{\phi_{j}\bigr\}$ contains a convergent subsequence.
Let $\bigl\{R_{k}\bigr\}$ be a monotone increasing sequence of positive numbers with $R_{k} \rightarrow +\infty$ as $k \rightarrow \infty.$ Therefore, we have for each $k,$ there is a number $j_{k}$ such that $4R_{k}< \rho_{j}$ whenever $j \geq j_{k}.$ Since $u_{d} \in L^{\infty}(\mathbb{R}^n) \cap \mathcal{L}_{s}(\mathbb{R}^n),$ we have $ \phi_{{j}} \in L^{\infty}(\mathbb{R}^n) \cap \mathcal{L}_{s}(\mathbb{R}^n)$ for each $j \geq 1.$ Now, we can use Theorem 1.4 \cite{Fall1} to get the following estimates:\\
For $0<s<\frac{1}{2},$ $ 2s+\epsilon<1 $
\begin{enumerate}
\item[i)] $4s+\epsilon<1,$ then $$ \norm{\phi_{{j}}}_{C^{0,4s+\epsilon} (B_{2R_{k}})} \leq C \Bigl(\norm{\phi_j}_{L^{\infty} (\mathbb{R}^{n})}  + \norm{\phi_{j}^{p}-\phi_j}_{C^{0,2s+\epsilon}(B_{4R_{k}})} \Bigr) $$
\item[ii)]   $1<4s+\epsilon<2,$ then $$ \norm{\phi_{{j}}}_{C^{1,4s+\epsilon-1} (B_{2R_{k}})} \leq C \Bigl(\norm{\phi_j}_{L^{\infty} (\mathbb{R}^{n})}  + \norm{\phi_{j}^{p}-\phi_j}_{C^{0,2s+\epsilon}(B_{4R_{k}})} \Bigr), $$
\end{enumerate}
and for $\frac{1}{2} \leq s <1,$ $2s+\epsilon-1<1$
\begin{enumerate}
	\item[iii)] $4s+\epsilon-1<1,$ then $$ \norm{\phi_{{j}}}_{C^{1,4s+\epsilon-1} (B_{2R_{k}}} \leq C \Bigl(\norm{\phi_j}_{L^{\infty} (\mathbb{R}^{n})}  + \norm{\phi_{j}^{p}-\phi_j}_{C^{1,2s+\epsilon-1}(B_{4R_{k}})} \Bigr) $$
	\item[iv)]   $1<4s+\epsilon-1<2,$ then $$ \norm{\phi_{{j}}}_{C^{2,4s+\epsilon-1} (B_{2R_{k}})} \leq C \Bigl(\norm{\phi_j}_{L^{\infty} (\mathbb{R}^{n})}  + \norm{\phi_{j}^{p}-\phi_j}_{C^{1,2s+\epsilon-1}(B_{4R_{k}})} \Bigr), $$
\end{enumerate}
where the constant $C>0$ is independent of $j.$

Let us recall the inequality (\ref{bound1.1}) here:
\begin{align*}
	d	\frac{c_{n,s}}{2} \int_{T(\Omega)} \frac{\abs{u_{d}(x)-u_{d}(y)}^{2}}{\abs{x-y}^{n+2s}}dxdy + 
	\int_{\Omega}u^{2}dx=\int_{\Omega} u_{d}^{p+1} \leq C_{0}d^{\frac{n}{2s}},
\end{align*}
where $C_{0}$ is independent of $d.$ This yields
\begin{align}
	\int_{B_{\rho_{j}}} \phi_{j}^{p+1} \leq C_{0},
\end{align}
and
\begin{align}\label{dist1.35}
	\norm{\phi_{j}}_{H^{s}(B_{\rho_{j}})} \leq C_{0}, \text{ for all } j\geq 1.
\end{align}
Also, by Theorem \ref{lr-estimate} we have $$\int_{\Omega} u_{d}^{r} \leq B(r)d^{\frac{n}{2s}} \,\,\, \text{ for all } r \geq 1$$ which implies that
\begin{align}\label{lr-estimate1}
	\int_{B_{\rho_{j}}} \phi_{j}^r \leq B(r),\,\,\,\text{ for all } j\geq 1 \text{ and } r \geq 1.
\end{align}
 By Lemma \ref{reg} and  Theorem \ref{bound1}, we have 
\begin{align}\label{dist1.2}
\norm{u_{d}}_{L^{\infty}(\mathbb{R}^{n})} \leq C_{1},
\end{align}
where the constant $C_{1}$ is independent of the diffusion constant $d.$ So the equations (\ref{lr-estimate1}), (\ref{dist1.2}) and Theorem 1.3 \cite{Fall1} imply that $$ \norm{\phi_{j}}_{X_{s}(\overline{B}_{R_{k}})} < C_2 \,\,\, \text{ for all } j \geq j_{k},$$ where the constant $C_2>0$ is independent of $j$ and the space $X_{s}(\overline{B}_{R_{k}})$ is identified with one of the spaces $C^{0,4s+\epsilon}(\overline{B}_{R_{k}}),$ $C^{1,4s+\epsilon-1}(\overline{B}_{R_{k}})$ or $C^{2,4s+\epsilon-1}(\overline{B}_{R_{k}})$ with same assumptions on $s$ and $\epsilon$ as above. Therefore $\bigl\{\phi_{j}\bigr\}$ is a relatively compact set in $X_{s}(\overline{B}_{R_{k}}),$ hence by the standard diagonal process, one can extract a convergent subsequence of $\bigl\{\phi_{j}\bigr\},$ we continue to denote such a subsequence by  $\bigl\{\phi_{j}\bigr\}$ itself such that  \begin{align*}
\phi_{j} \rightarrow v \,\,\, \text{in } C^{0, 2s+\epsilon}_{loc}(\mathbb{R}^{n}) \text{ when } 0<s<\frac{1}{2}, 2s+\epsilon<1 \end{align*} or

 \begin{align*} \phi_{j} \rightarrow v \,\,\, \text{in } C^{1, 2s+\epsilon-1}_{loc}(\mathbb{R}^{n}) \text{ when } \frac{1}{2}<s<1, 2s+\epsilon-1<1
\end{align*}
for some $v.$ The limit $v \in C^{0,2s+\epsilon}(\mathbb{R}^{n}) \cap H^{s}(\mathbb{R}^{n})$ when $0<s<\frac{1}{2}, 2s+\epsilon<1$ or $v \in C^{1,2s+\epsilon-1}(\mathbb{R}^{n}) \cap H^{s}(\mathbb{R}^{n})$ when $\frac{1}{2}<s<1, 2s+\epsilon-1<1$  follows from (\ref{dist1.35}). Consequently,
we have
$$\lim_{\abs{x}\rightarrow \infty}v(x)=0.$$ Using Theorem 1.1 \cite{Du}, we have $(-\Delta)^{s}\phi_{j}(x) $ converges to $ (-\Delta)^{s}v(x)$ point-wise in $\mathbb{R}^{n}.$ 
%For this, let us consider 
%\begin{align}
%(-\Delta)^{s}v(x)-(-\Delta)^{s}\phi_{j}(x)&= c_{n,s} \int_{\mathbb{R}^{n}} \left[\frac{(v(x)-v(y))-(\phi_{j}(x)-\phi_{j}(y))}{\abs{x-y}^{n+2s}} \right] dy \nonumber \\
%&=  c_{n,s} \int_{B_{R_{k}}} \frac{(v(x)-\phi_{j}(x))-(v(y)-\phi_{j}(y))}{\abs{x-y}^{n+2s}}dy + c_{n,s} \int_{\mathcal{C}B_{R_{k}}} \frac{(v(x)-\phi_{j}(x))-v(y)}{\abs{x-y}^{n+2s}}dy \nonumber \\ & \hspace{1cm} +  \int_{\mathcal{C}B_{R_{k}}} \frac{\phi_{j}(y)}{\abs{x-y}^{n+2s}}dy \nonumber \\
%&=F_{j,R_{k}} + G_{j,R_{k}}+ H_{j,R_{k}}.
%\end{align}
%Now, using $\phi_{j} \rightarrow v \,\,\, \text{in } C^{2}_{loc}(\mathbb{R}^{n})$ and Lebesgue dominated convergence theorem, we get
% $$\displaystyle \lim_{R_{k} \to \infty}\lim_{j \to \infty} F_{j,R_{k}}= \displaystyle \lim_{R_{k} \to \infty}\lim_{j\to \infty}G_{j,R_{k}}= \displaystyle \lim_{R_{k} \to \infty}\lim_{j\to \infty}H_{j,R_{k}}=0. $$
%$ \displaystyle \lim_{j \to \infty} F_{j,R_{k}}=0.$ Also,
%$$ \displaystyle \lim_{j\to \infty}G_{j,R_{k}}= - c_{n,s} \int_{\mathcal{C}B_{R_{k}}} \frac{v(y)}{\abs{x-y}^{n+2s}}dy.$$ Further, since $v$ is uniformly bounded in $\mathbb{R}^{n}$, we have
%$$\displaystyle \lim_{R_{k} \to \infty}\lim_{j\to \infty}G_{j,R_{k}}=0.$$ Similarly, $\phi_{j}$ is uniformly bounded implies that
%$$\displaystyle \lim_{R_{k} \to \infty}\lim_{j\to \infty}H_{j,R_{k}}=0.$$
Consequently, we see that the limit $v$ satisfies the equation
\begin{align}
(-\Delta)^{s}v+v=v^{p} \,\,\, \text{in } \mathbb{R}^{n}.
\end{align}
Clearly, $v \geq 0$ because each $\phi_{j} \geq 0.$ Since by Lemma \ref{sup1}, we have $\phi_{j}(0)=u_{d_{j}}(z_{j})>1$ for each $j \geq 1,$ one can see that $v \not \equiv 0.$ 

%$(\mathbb{R}^{n}) \leq C \bigl( \norm{u_{d_{j}}}_{L^{\infty}(\mathbb{R}^{n})}+ \norm{u_{d_{j}- u_{d_{j}}^{p}}_{L^{\infty}(\mathbb{R}^{n})}$

Using Theorem \ref{radial1}, one can see that $v$ is radially symmetric and decreasing about some point in $\mathbb{R}^{n}.$ Since $$\nabla v(0)= \lim_{j \rightarrow \infty} \nabla \phi_{j}(0)=0$$ so it implies that $v$ is radially symmetric about the origin. And by Theorem \ref{radial2}, $v$ has a power type of decay at infinity, i.e.,
\begin{align*}
v(r)\leq \frac{C_{2}}{r^{n+2s}}, \,\,\, r \geq 1.
\end{align*}
Now, we derive a lower bound on the critical value $c_{d_{j}}.$ Let us define \begin{align}\label{dist1.10}
\delta_{R}:= \frac{C_{2}}{R^{n+2s}}, 
\end{align}
where $R>0$ arbitrarily large real number. Then, there exists a positive integer $j_{R}$ such that if $j \geq j_{R}$ then $\rho_{j}\geq 2R$ and \begin{align}\label{dist1.6}
\norm{\phi_{j}-v}_{C^{2}(\overline{B}_{2R})} \leq \delta_{R}.
\end{align}
By Lemma \ref{critical2}, we have $$c_{d_{j}}= M[u_{d_{j}}]=J_{d_{j}}(u_{d_{j}}).$$ Using this fact and (\ref{bound1.19}), we obtain
 \begin{align}\label{dist1.26}
 c_{d_{j}}=& \, \left(\frac{1}{2}-\frac{1}{p+1}\right)\int_{\Omega}u_{d_{j}}^{p+1}dx  \\
\geq & \, \left(\frac{1}{2}-\frac{1}{p+1}\right) \int_{\abs{x-z_{j}}<d_{j}^{\frac{1}{2s}}R}u_{d_{j}}^{p+1}dx \nonumber \\
=& \, d_{j}^{\frac{n}{2s}} \left(\frac{1}{2}-\frac{1}{p+1}\right) \int_{\abs{y} < R} \phi_{j}^{p+1}dy. \nonumber 
\end{align}
Now, we write
\begin{align} \label{dist1.7}
c_{d_{j}}=& \, d_{j}^{\frac{n}{2s}} \left(\left(\frac{1}{2}-\frac{1}{p+1}\right)\int_{B_{R}} v^{p+1}dy+ F_{j} \right), 
\end{align}
where $$F_{j}:=  \left(\frac{1}{2}-\frac{1}{p+1}\right) \int_{B_{R}}\left( \phi_{{j}}^{p+1}- v^{p+1} \right)dy. $$
By Eq. (\ref{dist1.6}), we have for all $y \in B_{R},~ j \geq j_{R}$ 
\begin{align}
\abs{ \phi_{{j}}^{p+1}- v^{p+1}} \leq C \abs{\phi_{{j}}-v} \leq \delta_{R},
\end{align}
where $C>0$ is some constant. This implies that 
\begin{align*}
\abs{F_{j}} \leq \left(\frac{1}{2}-\frac{1}{p+1}\right)C \abs{B_{R}}  \delta_{R}=C_{3}R^{n}\delta_{R},
\end{align*}
where $$C_{3}= \left(\frac{1}{2}-\frac{1}{p+1}\right)\frac{w_{n}}{n}
C$$ and $w_{n}$ denotes the surface area of the unit sphere in $\mathbb{R}^{n}.$ Consequently, Inequality (\ref{dist1.7}) becomes
\begin{align}\label{dist1.8}
c_{d_{j}}\geq d_{j}^{\frac{n}{2s}} \left[\left(\frac{1}{2}-\frac{1}{p+1}\right)\int_{B_{R}} v^{p+1}dy-C_{3}R^{n}\delta_{R} \right].
\end{align}
Now, it is easy to see that
\begin{align}
\left(\frac{1}{2}-\frac{1}{p+1}\right)\int_{B_{R}} v^{p+1}dy= F(v)-\left(\frac{1}{2}-\frac{1}{p+1}\right)\int_{\abs{y}>R} v^{p+1}dy,
\end{align}
where $F(v)$ is defined earlier in (\ref{energy2}). Simplifying the second term on right hand side, we get 
\begin{align*}
\left(\frac{1}{2}-\frac{1}{p+1}\right)\int_{\abs{y}>R} v^{p+1}dy=& \, \left(\frac{1}{2}-\frac{1}{p+1}\right) \int_{R}^{\infty} \frac{r^{n-1}w_{n}}{r^{(n+2s)(p+1)}}dr \\
=& \, \left(\frac{1}{2}-\frac{1}{p+1}\right) \frac{w_{n}}{(n+2s)p+2s}\frac{1}{R^{(n+2s)p+2s}}=\frac{C_{4}}{R^{(n+2s)p+2s}}.
\end{align*}
Therefore, one can write

\begin{align}\label{dist1.9}
\left(\frac{1}{2}-\frac{1}{p+1}\right)\int_{B_{R}} v^{p+1}dy = F(v)-\frac{C_{4}}{R^{(n+2s)p+2s}}.
\end{align}
 On combining Equations (\ref{dist1.10}), (\ref{dist1.8}) and (\ref{dist1.9}), we get for $j \geq j_{R},$
\begin{align}\label{dist1.11}
c_{d_{j}} \geq & \, d_{j}^{\frac{n}{2s}}\left( F(v)- \frac{C_{4}}{R^{(n+2s)p+2s}}-\frac{C_{2}C_{3}}{R^{2s}} \right)  \nonumber \\
 \geq& \,  d_{j}^{\frac{n}{2s}}\left( F(v)- \frac{C_{5}}{R^{2s}} \right) ,
\end{align}
where $C_{5}$ is independent of $j$ and $R.$\\
Now, we derive an upper bound on the critical value $c_{d_{j}}.$ Without loss of generality, we may assume that the domain $\Omega$ is a subset of $\mathbb{R}^{n}_{+}$ and $0 \in \partial \Omega.$ Given Definition \ref{ground}, let $w$ be a ground state solution of (\ref{P2}). Define
\begin{align*}
\Omega_{d}:=\Bigl\{\frac{x}{d^{\frac{1}{2s}}} \mid x\in \Omega \Bigr\},\\
w_{d}(x):=w\left( \frac{x}{d^{\frac{1}{2s}}}\right), \text{ for } x\in \mathbb{R}^{n}.
\end{align*}
Since $w \geq 0$ so this implies that $w_{d} \geq 0$. Define 
$$g_{d}(t):=J_{d}(tw_{d}),\,\,\,\,t \geq 0.$$ Then by Lemma \ref{critical2}, there exists unique $t_{0}=t_{0}(d) >0$ at which $g_{d}$ attains maximum. Note that $t_{0}(d) \rightarrow 1$ as $d \downarrow 0.$ Hence, we have 
\begin{align*}
M[w_{d}]=& \, J_{d}\left(t_{0}w_{d}\right) \\
=& \, \frac{t_{0}^{2}}{2} \left[\frac{dc_{n,s}}{2}\int_{T(\Omega)}\frac{\abs{w_{d}(x)-w_{d}(y)}^{2}}{\abs{x-y}^{n+2s}}dxdy+ \int_{\Omega}w_{d}^{2}dx\right]-\frac{t_{0}^{p+1}}{p+1}\int_{\Omega}w_{d}^{p+1}dx \\
=& \, \frac{t_{0}^{2}}{2} \left[\frac{dc_{n,s}}{2}\int_{T(\Omega)}\frac{\abs{w\left(\frac{x}{d^{\frac{1}{2s}}}\right)-w\left(\frac{y}{d^{\frac{1}{2s}}}\right)}^{2}}{\abs{x-y}^{n+2s}}dxdy+ \int_{\Omega}w^{2}\left(\frac{x}{d^{\frac{1}{2s}}}\right)dx\right]-\frac{t_{0}^{p+1}}{p+1}\int_{\Omega}w^{p+1}\left(\frac{x}{d^{\frac{1}{2s}}}\right)dx. 
\end{align*}
The change of variables $$\frac{x}{d^{\frac{1}{2s}}}=a,\, \frac{y}{d^{\frac{1}{2s}}}=b,$$ gives us 
\begin{align*}
M[w_{d}]=& \, d^{\frac{n}{2s}} \left(\frac{t_{0}^{2}}{2} \left[\frac{c_{n,s}}{2}\int_{T(\Omega_{d})}\frac{\abs{w(a)-w(b)}^{2}}{\abs{a-b}^{n+2s}}dadb+ \int_{\Omega_{d}}w^{2}da\right]-\frac{t_{0}^{p+1}}{p+1}\int_{\Omega_{d}}w^{p+1}da \right) \\
=&\, d^{\frac{n}{2s}} \left( \frac{t_{0}^{2}}{2} \left[\frac{c_{n,s}}{2}\int_{\Omega_{d}} \int_{\Omega_{d}} \frac{\abs{w(a)-w(b)}^{2}}{\abs{a-b}^{n+2s}}dadb+ 2c_{n,s}\int_{\mathcal{C}\Omega_{d}} \int_{\Omega_{d}} \frac{\abs{w(a)-w(b)}^{2}}{\abs{a-b}^{n+2s}}dadb + \int_{\Omega_{d}}w^{2}da\right]-\frac{t_{0}^{p+1}}{p+1}\int_{\Omega_{d}}w^{p+1}da \right).
\end{align*}
Let us denote by $I_{d}:$ 
$$I_{d}=\frac{t_{0}^{2}}{2} \left[\frac{c_{n,s}}{2}\int_{\Omega_{d}} \int_{\Omega_{d}} \frac{\abs{w(a)-w(b)}^{2}}{\abs{a-b}^{n+2s}}dadb+ 2c_{n,s}\int_{\mathcal{C}\Omega_{d}} \int_{\Omega_{d}} \frac{\abs{w(a)-w(b)}^{2}}{\abs{a-b}^{n+2s}}dadb + \int_{\Omega_{d}}w^{2}da\right]-\frac{t_{0}^{p+1}}{p+1}\int_{\Omega_{d}}w^{p+1}da.$$
Since $$t_{0}(d) \rightarrow 1 \text{ as }d\downarrow 0,$$ we get
\begin{align*}
I_{d}= \frac{1}{2} \left[\frac{c_{n,s}}{2}\int_{\mathbb{R}^{n}_{+}} \int_{\mathbb{R}^{n}_{+}} \frac{\abs{w(a)-w(b)}^{2}}{\abs{a-b}^{n+2s}}dadb+ 2c_{n,s}\int_{\mathcal{C} \mathbb{R}^{n}_{+}} \int_{\mathbb{R}^{n}_{+}} \frac{\abs{w(a)-w(b)}^{2}}{\abs{a-b}^{n+2s}}dadb + \int_{\mathbb{R}^{n}_{+}}w^{2}da\right]-\frac{1}{p+1}\int_{\mathbb{R}^{n}_{+}}w^{p+1}da + o(1)
\end{align*}
as $d \downarrow 0.$
Further, $w$ is a nonnegative and radially symmetric implies that 
$$\int_{\mathbb{R}^{n}_{+}}w^{2}da=\frac{1}{2}\int_{\mathbb{R}^{n}}w^{2}da,\,\, \int_{\mathbb{R}^{n}_{+}}w^{p+1}da=\frac{1}{2}\int_{\mathbb{R}^{n}}w^{p+1}da,$$
$$\int_{\mathbb{R}^{n}_{+}} \int_{\mathbb{R}^{n}_{+}} \frac{\abs{w(a)-w(b)}^{2}}{\abs{a-b}^{n+2s}}dadb = \frac{1}{4} \int_{\mathbb{R}^{n}} \int_{\mathbb{R}^{n}} \frac{\abs{w(a)-w(b)}^{2}}{\abs{a-b}^{n+2s}}dadb,$$
$$ \int_{\mathcal{C}\mathbb{R}^{n}_{+}} \int_{\mathbb{R}^{n}_{+}} \frac{\abs{w(a)-w(b)}^{2}}{\abs{a-b}^{n+2s}}dadb = \frac{1}{4} \int_{\mathbb{R}^{n}} \int_{\mathbb{R}^{n}} \frac{\abs{w(a)-w(b)}^{2}}{\abs{a-b}^{n+2s}}dadb.$$
Using these estimates, we get
$$ I_{d} < \frac{1}{2}\left( \frac{1}{2} \left[ \frac{c_{n,s}}{2}\int_{\mathbb{R}^{n}} \int_{\mathbb{R}^{n}} \frac{\abs{w(a)-w(b)}^{2}}{\abs{a-b}^{n+2s}}dadb + \int_{\mathbb{R}^{n}}w^{2}da \right]- \frac{1}{p+1} \int_{\mathbb{R}^{n}}w^{p+1}da\right)+ o(1)= \frac{1}{2}F(w)+o(1),$$
as $d \downarrow 0.$
Thus, we have 
\begin{align}\label{dist1.12}
M[w_{d}]=d^{\frac{n}{2s}}I_{d}< \frac{d^{\frac{n}{2s}}}{2}F(w)+o(1),
\end{align}
as $d\downarrow 0.$
Using (c) of Theorem \ref{radial3}, we have $0<F(w) \leq F(v)$ for any nonnegative nonzero classical solution $v$ of (\ref{P2}) and by Lemma \ref{critical2}, we have
 $$c_{d_{j}} \leq M[w_{d_{j}}]< \frac{d_{j}^{\frac{n}{2s}}}{2}F(v)$$ for $d_{j}$ sufficiently small. By taking $R$ sufficiently large in (\ref{dist1.11}), we thus obtain a contradiction. This proves (A). 

\begin{rem}
In the classical case \cite{Ni1}, authors have defined diffeomorphisms which straightens a boundary portion near $Q \in \partial \Omega.$ Further, using scaling and translations of least energy solutions $u_{d}$ of \eqref{P3}, the classical problem (\ref{P3}) gets transferred into new elliptic equation. Due to the nonlocal nature of the fractional Laplacian and boundary condition in our problem, it is almost impossible to introduce such scaling and translation arguments.
\end{rem}

\noindent \textbf{Proof of (B).} Now, we claim that $z_{d} \in \partial \Omega.$ Suppose that there is a decreasing sequence $d_{j} \downarrow 0$ such that $z_{d_{j}}:=z_{j} \in \Omega.$ We have from Lemma \ref{dist} that the sequence $\left\{z_{j}\right\}$ converges to some $z \in \partial \Omega.$ Without loss of generality, let us assume that $z=0.$ Define
\[ \widehat{u}_{j}(x):= \begin{cases}
u_{d_{j} }(x)& \, \text{in $\mathbb{R}^{n}_{+}$}, \\
u_{d_{j} }(x', -x_{n}) & \, \text{in $\mathbb{R}^{n}_{-}$},
\end{cases} \] 
where $$x'=(x_{1},x_{2},\dots, x_{n-1}),\,\, \mathbb{R}^{n}_{+}=\bigl\{(x',x_{n}) \mid x_{n} \geq 0 \bigr\},\,\,\mathbb{R}^{n}_{-}=\bigl\{(x',x_{n}) \mid x_{n} \leq 0 \bigr\}.$$  Also, define a scaled function
\begin{align} \label{alpha}
\psi_{j}(y):=\widehat{u}_{j}\left(yd_{j}^{\frac{1}{2s}}+z_{j}\right)\,\,\,\text{for } y \in \mathbb{R}^{n}.
\end{align}
Note that for $z_{j}=(z_{j}',z_{jn}),$ we can write $z_{jn}=\alpha_{j}d_{j}^{\frac{1}{2s}}$ for some $\alpha_{j}>0.$ The sequence $\left\{\alpha_{j}\right \} $ is bounded, which follows from  Lemma \ref{dist}.
Let \begin{align}
\rho_{j}:=\frac{\rho(z_{j}, \partial \Omega)}{d_{j}^{\frac{1}{2s}}} ,
\end{align}
where $\rho (z_{j}, \partial \Omega)$ denotes the distance between $z_{j}$ and $ \partial \Omega.$
Note that the function $\psi_{j}$ satisfies the equation
\begin{align}\label{dist1.21}
(-\Delta)^{s}\psi_{j}(y)+ \psi_{j}(y)= \psi_{j}(y)^{p}+ d_{j}{h}(y) \,\,\, \text{in } B_{\rho_{j}},
\end{align}
for some function ${h}$ of $y.$  
To see this, let $y \in B_{\rho_{j}}$, we have
\begin{align}\label{dist1.13}
	(-\Delta)^{s}\psi_{j}(y)=&\,c_{n,s} P.V. \int_{\mathbb{R}^{n}}\frac{\psi_{j}(y)-\psi_{j}(x)}{\abs{y-x}^{n+2s}}dx= c_{n,s} \lim_{\epsilon \rightarrow 0} \int_{\mathcal{C} B_{\epsilon}(y)}\frac{\psi_{j}(y)-\psi_{j}(x)}{\abs{y-x}^{n+2s}}dx  \nonumber \\
	=&\, c_{n,s} \lim_{\epsilon \rightarrow 0} \left[  \int_{\mathcal{C} B_{\epsilon}(y)}\frac{\widehat{u}_{j}(yd_{j}^{\frac{1}{2s}}+z_{j})-\widehat{u}_{j}(xd_{j}^{\frac{1}{2s}}+z_{j})}{\abs{y-x}^{n+2s}}dx  \right] \nonumber \\
	=&\, c_{n,s} \lim_{\epsilon \rightarrow 0} \Biggl[  \int_{\bigl\{x_{n} \geq -\alpha_{j}\bigr\}\bigcap \mathcal{C} B_{\epsilon}(y)}\frac{\widehat{u}_{j}(yd_{j}^{\frac{1}{2s}}+z_{j})-\widehat{u}_{j}(xd_{j}^{\frac{1}{2s}}+z_{j})}{\abs{y-x}^{n+2s}}dx \nonumber \\    & \, \hspace{0.3cm}+\int_{\bigl\{x_{n} \leq-\alpha_{j}\bigr\}\bigcap \mathcal{C}B_{\epsilon}(y)}\frac{\widehat{u}_{j}(yd_{j}^{\frac{1}{2s}}+z_{j})-\widehat{u}_{j}(xd_{j}^{\frac{1}{2s}}+z_{j})}{\abs{y-x}^{n+2s}}dx \Biggr].
\end{align}
For $y_{n} \geq -\alpha_{j},$ we have 
\begin{align}\label{dist1.14}
	(-\Delta)^{s}\psi_{j}(y)= &\, c_{n,s}\lim_{\epsilon \rightarrow 0} \Biggl[  \int_{\bigl\{x_{n} \geq -\alpha_{j}\bigr\}\bigcap \mathcal{C} B_{\epsilon}(y)}\frac{{u}_{{j}}(yd_{j}^{\frac{1}{2s}}+z_{j})-{u}_{{j}}(xd_{j}^{\frac{1}{2s}}+z_{j})}{\abs{y-x}^{n+2s}}dx  \nonumber \\ 
	& \,+ \,\,\, \int_{\bigl\{x_{n} \leq-\alpha_{j}\bigr\}\bigcap \mathcal{C} B_{\epsilon}(y)}\frac{{u}_{{j}}(yd_{j}^{\frac{1}{2s}}+z_{j})-{u}_{{j}}(x'd_{j}^{\frac{1}{2s}}+z_{j}', -(x_{n}+\alpha_{j}d_{j}^{\frac{1}{2s}}))}{\abs{y-x}^{n+2s}}dx \Biggr] \nonumber \\ 
	=&\, c_{n,s} \lim_{\epsilon \rightarrow 0} \Biggl[  \int_{\bigl\{x_{n} \geq -\alpha_{j}\bigr\} \bigcap \mathcal{C} B_{\epsilon}(y)}\frac{{u}_{{j}}(yd_{j}^{\frac{1}{2s}}+z_{j})-{u}_{{j}}(xd_{j}^{\frac{1}{2s}}+z_{j})}{\abs{y-x}^{n+2s}}dx  \nonumber \\ 
	&\, \hspace{0.3cm} + \,\,\, \int_{\bigl\{x_{n} \leq-\alpha_{j}\bigr\}\bigcap \mathcal{C} B_{\epsilon}(y)}\frac{{u}_{{j}}(yd_{j}^{\frac{1}{2s}}+z_{j})-{u}_{{j}}(xd_{j}^{\frac{1}{2s}}+z_{j})+ {u}_{{j}}(xd_{j}^{\frac{1}{2s}}+z_{j})   -{u}_{{j}}(x'd_{j}^{\frac{1}{2s}}+z_{j}', -(x_{n}+\alpha_{j}d_{j}^{\frac{1}{2s}}))}{\abs{y-x}^{n+2s}}dx \Biggr] \nonumber \\ 
	=&\, c_{n,s} \lim_{\epsilon \rightarrow 0} \Biggl[ \int_{\mathcal{C} B_{\epsilon}(y)}\frac{{u}_{{j}}(yd_{j}^{\frac{1}{2s}}+z_{j})-{u}_{{j}}(xd_{j}^{\frac{1}{2s}}+z_{j})}{\abs{y-x}^{n+2s}}dx  \nonumber \\
	&\,\hspace{0.3cm}+  \int_{\bigl\{x_{n} \leq-\alpha_{j}\bigr\} \bigcap \mathcal{C}B_{\epsilon}(y)}\frac{{u}_{{j}}(xd_{j}^{\frac{1}{2s}}+z_{j})   -{u}_{{j}}(x'd_{j}^{\frac{1}{2s}}+z_{j}', -(x_{n}+\alpha_{j}d_{j}^{\frac{1}{2s}}))}{\abs{y-x}^{n+2s}}dx \Biggr] \nonumber \\
	=&\, c_{n,s}  \lim_{\epsilon \rightarrow 0} \left[ \int_{\mathcal{C} B_{\epsilon}(y)}\frac{{u}_{{j}}(yd_{j}^{\frac{1}{2s}}+z_{j})-{u}_{{j}}(xd_{j}^{\frac{1}{2s}}+z_{j})}{\abs{y-x}^{n+2s}}dx +  \int_{\bigl\{x_{n} \leq-\alpha_{j}\bigr\}\bigcap \mathcal{C} B_{\epsilon}(y)}\frac{{u}_{{j}}(xd_{j}^{\frac{1}{2s}}+z_{j})   -\widehat{u}_{{j}}(xd_{j}^{\frac{1}{2s}}+z_{j})}{\abs{y-x}^{n+2s}}dx \right].
\end{align}
Making change of variables $$yd_{j}^{\frac{1}{2s}}+z_{j}=a \text{ and } xd_{j}^{\frac{1}{2s}}+z_{j}=b,$$ we get 
\begin{align}\label{dist1.15}
	(-\Delta)^{s}\psi_{j}(y)=&\, d_{j}(-\Delta)^{s}u_{j}(a)+ d_{j} c_{n,s} \lim_{\eta \rightarrow 0} \int_{\bigl\{b_{n} \leq 0\bigr\} \bigcap \mathcal{C} B_{\eta }(a)}\frac{{u}_{{j}}(b)   -\widehat{u}_{{j}}(b)}{\abs{a-b}^{n+2s}}db \nonumber\\
	=&\, d_{j}(-\Delta)^{s}u_{j}(a)+ d_{j}h(a),
\end{align}
where $$\eta = \epsilon d_{j}^{\frac{1}{2s}} $$ and $$h(a)= c_{n,s} \lim_{\eta \rightarrow 0} \int_{\bigl\{b_{n} \leq 0\bigr\} \bigcap \mathcal{C} B_{\eta }(a)}\frac{{u}_{{j}}(b)   -\widehat{u}_{{j}}(b)}{\abs{a-b}^{n+2s}}db.$$ Note that $a \in \Omega.$\\
Now, consider the case $y_{n} \leq -\alpha_{j}.$ Equation (\ref{dist1.13}) becomes
\begin{align}\label{dist1.16}
	(-\Delta)^{s}\psi_{j}(y) = &\, c_{n,s} \lim_{\epsilon \rightarrow 0}\Biggl[  \int_{\bigl\{x_{n} \geq -\alpha_{j}\bigr\}\bigcap \mathcal{C} B_{\epsilon}(y)}\frac{{u}_{j}(y'd_{j}^{\frac{1}{2s}}+z_{j}', -(y_{n}d_{j}^{\frac{1}{2s}}+\alpha_{j}d_{j}^{\frac{1}{2s}}))-{u}_{j}(xd_{j}^{\frac{1}{2s}}+z_{j})}{\abs{y-x}^{n+2s}}dx \nonumber \\
	&\,+  \int_{\bigl\{x_{n} \leq-\alpha_{j}\bigr\}\bigcap \mathcal{C} B_{\epsilon}(y)}\frac{{u}_{j}(y'd_{j}^{\frac{1}{2s}}+z_{j}', -(y_{n}d_{j}^{\frac{1}{2s}}+\alpha_{j}d_{j}^{\frac{1}{2s}}))- {u}_{j}(x'd_{j}^{\frac{1}{2s}}+z_{j}', -(x_{n}d_{j}^{\frac{1}{2s}}+\alpha_{j}d_{j}^{\frac{1}{2s}}))}{\abs{y-x}^{n+2s}}dx  \Biggr] \nonumber \\
	=&\, I_{1}+ I_{2}, 
\end{align}
where $I_{1}$ and $I_{2}$ denote the first and second integral on the right hand side, respectively.
Let us introduce some of notations. We write $\widehat{x}= (x',-x_{n})$, $\widetilde{x}=(\widehat{x}',\widehat{x}_{n})$ and $\widehat{x}_{n}=-x_{n}$ for $x=(x',x_{n}) \in \mathbb{R}^{n}, n>1.$ Using these, let us compute
\begin{align}\label{dist1.18}
	I_{2}=&\, c_{n,s} \lim_{\epsilon \rightarrow 0} \Biggl[\int_{\bigl\{\widehat{x}_{n} \geq \alpha_{j}\bigr\} \bigcap \mathcal{C} B_{\epsilon}(\widehat{y})}\frac{{u}_{j}(\widehat{y}'d_{j}^{\frac{1}{2s}}+\hat{z}_{j}', \widehat{y}_{n}d_{j}^{\frac{1}{2s}}+\widehat{\alpha}_{j}d_{j}^{\frac{1}{2s}})- {u}_{j}(\widehat{x}'d_{j}^{\frac{1}{2s}}+\widehat{z}_{j}', \widehat{x}_{n}d_{j}^{\frac{1}{2s}}+\widehat{\alpha}_{j}d_{j}^{\frac{1}{2s}})}{\abs{\widehat{y}-\widehat{x}}^{n+2s}}d\widetilde{x}  \Biggr] \nonumber \\
	=&\, c_{n,s} \lim_{\epsilon \rightarrow 0} \Biggl[\int_{\bigl\{\widehat{x}_{n} \geq \alpha_{j}\bigr\} \bigcap \mathcal{C} B_{\epsilon}(\widehat{y})}\frac{{u}_{j}(\widetilde{y}d_{j}^{\frac{1}{2s}}+\widetilde{z}_{j})- {u}_{j}(\widetilde{x}d_{j}^{\frac{1}{2s}}+\widetilde{z}_{j} )}{\abs{\widehat{y}-\widehat{x}}^{n+2s}}d\widetilde{x}  \Biggr] \nonumber\\ 
	=&\, c_{n,s} \lim_{\epsilon \rightarrow 0} \Biggl[\int_{\bigl\{\widehat{x}_{n} \geq \alpha_{j}\bigr\} \bigcap \mathcal{C} B_{\epsilon}(\widetilde{y})}\frac{{u}_{j}(\widetilde{y}d_{j}^{\frac{1}{2s}}+\widetilde{z}_{j})- {u}_{j}(\widetilde{x}d_{j}^{\frac{1}{2s}}+\widetilde{z}_{j} )}{\abs{\widetilde{y}-\widetilde{x}}^{n+2s}}d\widetilde{x}  \Biggr].
\end{align}
Now, we simplify $I_{1}:$
\begin{align}\label{dist1.19}
	I_{1}=& \, c_{n,s} \lim_{\epsilon \rightarrow 0} \Biggl[ \int_{\bigl\{x_{n} \geq -\alpha_{j}\bigr\} \bigcap \mathcal{C} B_{\epsilon}({y})}\frac{{u}_{j}(y'd_{j}^{\frac{1}{2s}}+z_{j}', -(y_{n}d_{j}^{\frac{1}{2s}}+\alpha_{j}d_{j}^{\frac{1}{2s}}))- {u}_{j}(x'd_{j}^{\frac{1}{2s}}+z_{j}', -(x_{n}d_{j}^{\frac{1}{2s}}+\alpha_{j}d_{j}^{\frac{1}{2s}}))}{\abs{y-x}^{n+2s}} \nonumber \\ & \hspace{0.3cm}+ \int_{\{x_{n} \geq -\alpha_{j}\} \cap \mathcal{C} B_{\epsilon}({y})} \frac{ {u}_{j}(w'd_{j}^{\frac{1}{2s}}+z_{j}', -(x_{n}d_{j}^{\frac{1}{2s}}+\alpha_{j}d_{j}^{\frac{1}{2s}}))-{u}_{j}(xd_{j}^{\frac{1}{2s}}+z_{j})}{\abs{y-x}^{n+2s}}dx \Biggr] \nonumber \\
	=&\, c_{n,s}\lim_{\epsilon \rightarrow 0} \Biggl[ \int_{\bigl\{\widehat{x}_{n} \leq \alpha_{j}\bigr\}  \bigcap \mathcal{C}B_{\epsilon}(\widetilde{y})}\frac{{u}_{j}(\widetilde{y}d_{j}^{\frac{1}{2s}}+\widetilde{z}_{j})- {u}_{j}(\widetilde{x}d_{j}^{\frac{1}{2s}}+\widetilde{z}_{j} )}{\abs{\widetilde{y}-\widetilde{x}}^{n+2s}}d\widetilde{x}+ \int_{\bigl\{{x}_{n} \geq -\alpha_{j}\bigr\} \cap \mathcal{C}B_{\epsilon}(\widetilde{y})}\frac{{u}_{j}(\widetilde{x}d_{j}^{\frac{1}{2s}}+\widetilde{z}_{j})- \widehat{{u}}_{j}(\widetilde{x}d_{j}^{\frac{1}{2s}}+\widetilde{z}_{j} )}{\abs{\widetilde{y}-\widetilde{x}}^{n+2s}}d\widetilde{x} \Biggr].
\end{align}
Using these estimates for $I_{1}$ and $I_{2}$ in equation (\ref{dist1.16}), we get
\begin{align}
	(-\Delta)^{s}\psi_{j}(y)= c_{n,s} P.V. \int_{\mathbb{R}^{n}}\frac{{u}_{j}(\widetilde{y}d_{j}^{\frac{1}{2s}}+\widetilde{z}_{j})- {u}_{j}(\widetilde{x}d_{j}^{\frac{1}{2s}}+\widetilde{z}_{j} )}{\abs{\widetilde{y}-\widetilde{x}}^{n+2s}}d\widetilde{x} + c_{n,s} \lim_{\epsilon \rightarrow 0} \int_{\bigl\{{x}_{n} \geq -\alpha_{j}\bigr\} \bigcap \mathcal{C}B_{\epsilon}(\widetilde{y})}\frac{{u}_{j}(\widetilde{x}d_{j}^{\frac{1}{2s}}+\widetilde{z}_{j})- \widehat{{u}}_{j}(\widetilde{x}d_{j}^{\frac{1}{2s}}+\widetilde{z}_{j} )}{\abs{\widetilde{y}-\widetilde{x}}^{n+2s}}d\widetilde{x} .
\end{align}
By the change of variables $$\widetilde{y}d_{j}^{\frac{1}{2s}}+\widetilde{z}_{j}=e \text{ and } \widetilde{w}d_{j}^{\frac{1}{2s}}+\widetilde{z}_{j}=f,$$ we get
\begin{align}\label{dist1.20}
	(-\Delta)^{s}\psi_{j}(y)=&\, d_{j}(-\Delta)^{s}u_{j}(e)+ d_{j} c_{n,s} \lim_{\eta \rightarrow 0} \int_{\bigl\{f_{n} \leq 0\bigr\}\bigcap \mathcal{C}B_{\eta}(e)}\frac{{u}_{{j}}(f)   -\widehat{u}_{{j}}(f)}{\abs{e-f}^{n+2s}}df, \,\mbox{where} 
	f_n\,\,\mbox{is the $n$-th coordinate of} \,\,f \nonumber \\
	=&\, d_{j}(-\Delta)^{s}u_{j}(e) + d_{j}h(e).
\end{align}
Note that $e \in \Omega.$
%Hence, for $y \in B_{\rho_{j}}$ we get from equations (\ref{dist1.15}) and (\ref{dist1.20})
%\begin{align}\label{dist1.21}
%(-\Delta)^{s}\psi_{j}(y)= d_{j}(-\Delta)^{s}u_{j}(x) + d_{j}h(x),
%\end{align}
Further, for $y \in B_{\rho_{j}},$ we have
\[ \psi_{j}(y)= \widehat{u}_{j}(yd_{j}^{\frac{1}{2s}}+ z_{j})= \begin{cases}
	u_{d_{j} }(yd_{j}^{\frac{1}{2s}}+z_{j})& \text{if $ y_{n} \geq - \alpha_{j}$}, \\
	u_{d_{j} }(y'd_{j}^{\frac{1}{2s}}+z_{j}', -(y_{n}d_{j}^{\frac{1}{2s}}+ \alpha_{j}d_{j}^{\frac{1}{2s}})) &  \text{if $y_{n} \leq -\alpha_{j}$}.
\end{cases} \] 
We can write $$(y'd_{j}^{\frac{1}{2s}}+z_{j}', -(y_{n}d_{j}^{\frac{1}{2s}}+ \alpha_{j}d_{j}^{\frac{1}{2s}}))=\widetilde{y}d_{j}^{\frac{1}{2s}}+\widetilde{z}_{j}.$$ Again renaming the variables $yd_{j}^{\frac{1}{2s}}+ z_{j}$ and $\widetilde{y}d_{j}^{\frac{1}{2s}}+\widetilde{z}_{j}$ by $a$ and $e,$ respectively, we get
\[ \psi_{j}(y)= \begin{cases}
	u_{d_{j} }(a)& \text{if $ y_{n} \geq -\alpha_{j}$}, \\
	u_{d_{j} }(e) &  \text{if $y_{n} \leq -\alpha_{j}$}.
\end{cases} \] 
We know that $u_{j}$ satisfies the equation $(\ref{P1})$ in the point-wise sense as well. Therefore, combining above equation with the equations (\ref{dist1.15}), (\ref{dist1.20}), we have for $y \in B_{\rho_{j}}, $
\begin{align}
(-\Delta)^{s}\psi_{j}(y)+\psi_{j}(y)=\psi_{j}(y)^{p}+ d_{j}{h}(y).
\end{align}
Now, arguing as in the proof of (A) with minor modifications, one can obtain a convergent subsequence of $\bigl\{\psi_{j} \bigr\},$ which we denote again by $\bigl\{\psi_{j} \bigr\}$ such that $\psi_{j} \rightarrow v$ in $C_{loc}^{2}(\mathbb{R}^{n}).$ Therefore as $d_{j}\downarrow 0,$ one obtain 
\begin{align}
(-\Delta)^{s}v+ v= v^{p}\,\,\, \text{in } \mathbb{R}^{n}.
\end{align}
Since $v \in H^{s}(\mathbb{R}^{n})$ and $v$ is radially decreasing so that $v$ is spherically symmetric to $y=0.$ Moreover, $v$ has the power type decay at infinity, which follows from Theorem \ref{radial2}, i.e.,
\begin{align}
v(r) \leq \frac{C_{2}}{r^{n+2s}}, \,\,\, r \geq 1,
\end{align}
for some constant $C_{2}>0.$ Let us define $\delta_{R}$ as in (\ref{dist1.10}), i.e., $$\delta_{R}:= \frac{C_{2}}{R^{n+2s}},$$ for $R$ sufficiently large to be defined later. Then, there exists an integer $j_{R}$ such that for $j \geq j_{R},$
\begin{align}\label{dist1.23}
\norm{\psi_{j}-v}_{C^{2}(\overline{B_{4R}})} \leq \delta_{R}. 
\end{align}
We choose $R$ sufficiently large that $R > \alpha_{j}$ for all $j,$ where $\alpha_{j}$'s are same as defined earlier right after the equation (\ref{alpha}). We can choose such a $R$ because $\left\{\alpha_{j}\right\}$ is a bounded sequence. The following lemma is very useful to prove our claim that $z_{d} \in \partial \Omega.$ 
\begin{lem} (see Lemma 4.2 \cite{Ni1})\label{dist1.24}
Let $f \in C^{2}(\overline{B_{t}})$ be a radial function. Assume that $f$ satisfies $f'(0)=0$ and $f''<0$ for $0 \leq r \leq t.$ Then there exists a $\eta> 0$ such that if $g \in C^{2}(\overline{B_{t}}) $  satisfies \begin{enumerate}
\item $\nabla g(0)=0$ 
\item $ \norm{f-g}_{C^{2}(\overline{B_{t}})}< \eta,$
\end{enumerate}
then $\nabla g \neq 0$ for $x \neq 0.$
\end{lem}
Now, we use this lemma to show that $\psi_{j}$ has only one local maximum point in $B_{R}.$ For this, we choose two numbers $k,l$ $(0<k<l)$ such that $v''(r)<0 $ for $0 \leq r \leq k.$ Note that $v''(0)<0$ and $ v(k)<1.$ Let us define 
\begin{align*}
\theta = \min \bigl\{ \abs{v'(r)} \mid k \leq r \leq l \bigr\}.
\end{align*}
It is easy to observe that $\theta >0$ because $v'<0$ for $r>0.$ Then for $\delta_{R}<\theta,$ we have by (\ref{dist1.23}) that 
\begin{align*}
0< \theta - \delta_{R} \leq \abs{\nabla v(y)} -\abs{\nabla \psi_{j}(y)-\nabla v(y)} \leq \abs{\nabla \psi_{j}(y)} \text{ for }  k \leq \abs{y} \leq l.
\end{align*}
Applying Lemma \ref{dist1.24} in the ball $\overline{B}_{k},$ we conclude that $y=0$ is the only local maximum point of $\psi_{j}$ in $B_{l}.$ If $y_{j}$ is a maximum point of $\psi_{j}$ in $B_{R},$ then by Lemma \ref{sup1}, we have $\psi_{j} \geq 1.$ Choose $R>0$ sufficiently large so that $\delta_{R}< 1-v(l).$ Therefore $$\psi_{j}(y) \leq v(y) + \delta_{R} \leq v(l)+ \delta_{R}<1.$$ Hence $y_{j} \in B_{l},$ and therefore $y_{j}=0.$\\
If $\alpha_{j}>0,$ then by the definition of $\widehat{u}_{j},$ $z_{R}^{*}=(z_{j}', -\alpha_{j}d_{j}^{\frac{1}{2s}})$ is also a maximum point of $\widehat{u}_{j}.$ This implies that $(0,-\alpha_{j})$ is another maximum point of $\psi_{j}$ in $B_{R},$ a contradiction. This proves our claim. 
\qed

\section*{Appendix A}
\noindent \textbf{Proof of the Inequality (\ref{bound1.3}):} We show that for real numbers $x,y \geq 0$ and $k \geq 1,$
\begin{align}
	\frac{1}{k}(x^{k}-y^{k})^{2}\leq (x-y)(x^{2k-1}-y^{2k-1}).
\end{align}
Clearly, the inequality holds when either $x$ or $y$ or both are zero. Thus, without  loss of generality we may assume that $x>y>0.$ Now our claim is reduced to show that 
$$	\frac{1}{k} \left(1- \left(\frac{y}{x}\right)^k \right)^2 \leq \left( 1- \frac{y}{x}\right) \left(1-  \left(\frac{y}{x}\right)^{2k-1}\right)$$
i.e., to show that $$(1-a^k)^2 \leq k (1-a)(1-a^{2k-1}),$$ where $0<a:= \frac{y}{x}<1.$ Consider
\begin{align*}
f(a)&:= k (1-a)(1-a^{2k-1})- (1-a^k)^{2} \\
& \geq (1-a^k)  \left(k(1-a)- (1-a^k) \right)\\
& \geq (1-a^k)(1-a) \left(k-(1+a+a^2+\dots + a^{k-1}) \right)\\
& \geq (1-a^k)(1-a) (k-k) =0.
\end{align*}
This proves the inequality \qed.

%$$ \frac{k^2}{\left(1-a^k \right)^2} \geq  \left(\frac{1}{1-a}\right) \left(\frac{1}{1-a^{2k-1}} \right), $$
%where $0<a:= \frac{y}{x}<1.$ We know,
%$$\frac{1}{1-a}= 1+ a + a^2+ \dots $$
%$$ \left(\frac{1}{1-a^k} \right)^2= \left( 1+ a^k + a^{2k}+ \dots  \right)^2 $$
%$$\frac{1}{1-a^{2k-1}} $$

\section*{Acknowledgement} 
The first author thanks CSIR for the financial support under the scheme 09/1031(0009)/2019-EMR-I. The second
author thanks DST/SERB for the financial support under the grant no. CRG/2020/000041.%\newpage

\section*{Statement} 
On behalf of all authors, the corresponding author states that there is no conflict of interest.

\end{document}